\documentclass[12pt]{amsart}
\usepackage[headings]{fullpage}
\usepackage{amssymb,epic,eepic,epsfig}
\usepackage{graphicx,wrapfig}
\usepackage{verbatim}
\usepackage{texdraw}
\usepackage{url}
\usepackage[bookmarks=true,%
    colorlinks=true,%
    linkcolor=blue,%
    citecolor=blue,%
    filecolor=blue,%
    menucolor=blue,%
    urlcolor=blue,%
    breaklinks=true]{hyperref}
\usepackage{bm}
\usepackage[all]{xy}
    
\newcommand{\lbl}{\label}
\renewcommand{\eta}{\nu}

\newcommand{\Reid}{\mathrm{R}}

\newtheorem{theorem}{Theorem}[section]
\theoremstyle{definition}
\newtheorem{proposition}[theorem]{Proposition}
\newtheorem{lemma}[theorem]{Lemma}
\newtheorem{definition}[theorem]{Definition}
\newtheorem{remark}[theorem]{Remark}
\newtheorem{corollary}[theorem]{Corollary}
\newtheorem{conjecture}[theorem]{Conjecture}
\newtheorem{question}[theorem]{Question}

\def\BZ{\mathbb Z}

\def\BQ{\mathbb Q}
\def\BR{\mathbb R}
\def\BC{\mathbb C}

\def\calT{\mathcal T}
\def\calX{\mathcal X}

\def\ga{\gamma}

\def\la{\langle}
\def\ra{\rangle}

\def\longto{\longrightarrow}

\def\SL{\mathrm{SL}}
\def\PSL{\mathrm{PSL}}
\def\pt{\partial}

\def\coeff{\mathrm{coeff}}

\def\be{  \begin{equation} }
\def\ee{  \end{equation} }

\def\Hom{\mathrm{Hom}}

\def\diag{\text{diag}}
\def\hb{\hbar}
\def\Sp{\mathrm{Sp}}
\def\ISp{\mathrm{ISp}}

\def\vol{\mathrm{Vol}}
\def\geom{\mathrm{geom}}

\def\EP{\mathrm{EP}}
\def\cxymatrix#1{\xy*[c]\xybox{\xymatrix#1}\endxy}

\newcommand{\CZ}{\mathcal{Z}}
\newcommand{\CH}{\mathcal{H}}
\newcommand{\CT}{\mathcal{T}}
\newcommand{\CL}{\mathcal{L}}
\newcommand{\CP}{\mathcal{P}}
\newcommand{\C}{\mathbb{C}}
\newcommand{\Z}{\mathbb{Z}}
\newcommand{\Li}{{\rm Li}}

\newcommand{\wt}{\widetilde}
\newcommand{\eg}{\emph{e.g.}}
\newcommand{\ie}{\emph{i.e.}}
\newcommand{\cf}{\emph{cf.}}
\newcommand{\mb}{\mathbf}
\newcommand{\rank}{\mathrm{rank}}
\newcommand{\pd}{\partial}

\numberwithin{equation}{section}

\begin{document}


\title[The quantum content of the gluing equations]{
The quantum content of the gluing equations}
\author{Tudor Dimofte}
\address{School of Natural Sciences \\
         Institute for Advanced Study  \\
         Princeton, NJ 08540, USA \newline
         \indent Trinity College \\ Cambridge CB2 1TQ \\ UK
        }
\email{tdd@ias.edu}
\author{Stavros Garoufalidis}
\address{School of Mathematics \\
         Georgia Institute of Technology \\
         Atlanta, GA 30332-0160, USA \newline
         {\tt \url{http://www.math.gatech.edu/~stavros }}}
\email{stavros@math.gatech.edu}
\thanks{The work of SG is supported by NSF grant DMS-11-05678. The work of 
TD is supported primarily by the Friends of the Institute for Advanced Study, 
and in part by DOE grant DE-FG02-90ER40542. \\
\newline
1991 {\em Mathematics Classification.} Primary 57N10. Secondary 57M25.
}

\date{October 25, 2012}

\begin{abstract}
The gluing equations of a cusped hyperbolic 3-manifold $M$ 
are a system of polynomial equations in the shapes of an ideal triangulation
$\calT$ of $M$ that describe the complete hyperbolic structure of $M$ and 
its deformations. Given a Neumann-Zagier datum  (comprising the shapes together
with the gluing equations in a particular canonical form) 
we define a formal power series with coefficients in the invariant 
trace field of $M$ that should (a) agree with the asymptotic 
expansion of the Kashaev invariant to all orders, and (b) contain the nonabelian Reidemeister-Ray-Singer torsion of $M$ as its first subleading ``1-loop'' term. 
As a case study, we prove topological invariance of the 1-loop part of the
constructed series and extend it into a formal power series of rational
functions on the $\PSL(2,\BC)$ character variety of $M$. We provide a computer 
implementation of the first three terms of the series using the standard 
{\tt SnapPy} toolbox and check numerically the agreement of our torsion
with the Reidemeister-Ray-Singer for all $59924$ hyperbolic knots with at 
most 14 crossings.
Finally, we explain how the definition of our series follows from
the quantization of 3d hyperbolic geometry, using principles of
Topological Quantum Field Theory.
Our results have a straightforward extension to any 3-manifold $M$ with 
torus boundary components (not necessarily hyperbolic) that admits a regular 
ideal triangulation with respect to some $\PSL(2,\BC)$ representation.
\end{abstract}

\maketitle

\tableofcontents

\section{Introduction}
\lbl{sec.intro}

\subsection{The Kashaev invariant and perturbative 
Chern-Simons theory}
\lbl{sub.summary}

The {\em Kashaev invariant} $\la K \ra_N \in \BC$ 
of a knot $K$ in 3-space (for $N=2,3,\dots$) is a powerful sequence of 
complex numbers determined by the {\em Jones polynomial} of the knot \cite{Jo} 
and its cablings \cite{Turaev, Witten-Jones}. The {\em Volume Conjecture} 
of Kashaev and Murakami-Murakami \cite{K94,K95,MM} relates the 
Kashaev invariant of a hyperbolic knot $K$ with the hyperbolic volume 
$\vol(M)$ of its complement $M=S^3\setminus K$ \cite{Th}
\be
\lim_{N\to\infty} \frac{1}{N} \log|\la K \ra_N|=\frac{\vol(M)}{2\pi} \,.
\ee
A generalization  of the Volume Conjecture \cite{Gu} 
predicts a full asymptotic expansion of the Kashaev invariant to all orders in
$1/N$ 
\be
\lbl{eq.KAS}
\la K \ra_N \overset{N\to\infty}{\sim} \CZ_M(2 \pi i /N)
\ee
for a suitable formal power series 
\be
\lbl{eq.ZMhb}
\qquad \CZ_M(\hb)=\exp\bigg(\frac1\hbar S_{M,0}-\frac32\log\hbar+S_{M,1}
+\sum_{n\geq 2}\hbar^{n-1}S_{M,n}\bigg)\,,\qquad \hbar = \frac{2\pi i}{N}\,.
\ee
The formal power series $\CZ_M(\hb)$ in \eqref{eq.ZMhb} is 
conjectured to coincide with the perturbative partition function of Chern-Simons theory with complex gauge group $\SL(2,\BC)$ along the discrete
faithful representation $\rho_0$ of the hyperbolic manifold $M$.
Combining such an interpretation with
further conjectures of \cite{DGLZ,Ga1,GL} one predicts that 
\begin{itemize}
\item
$S_{M,0}=i(\text{Vol}_M+i\text{CS}_M) \in \BC/(4 \pi^2\BZ)$ is the complexified volume of $M$ (\cf\ \cite{Thurston-paper, Neumann-combi}).
\item
$S_{M,1}$ is related \cite{Witten-cx, BNW, GuM} to the nonabelian 
Ray-Singer torsion \cite{RS}, which ought to equal (\cf\ \cite{Muller-cx}) the combinatorial
nonabelian Reidemeister torsion. More precisely, 
\cite[Conj.1.8]{DG} we should have
\be \label{TRMS1}
\tau^\Reid_M= 4\pi^3 \exp(-2 S_{M,1}) \in E_M^*
\ee
where $\tau^\Reid_M$ is the nonabelian Reidemeister-Ray-Singer torsion of 
$M$ with respect to the meridian \cite{Po,Db}, 
and $E_M$ is the invariant trace field of $M$.
\item
For $n \geq 2$, the {\em $n$-loop invariants} $S_{M,n}$ 
are conjectured to lie in the invariant trace field $E_M$ \cite{DGLZ,Ga1}.
\end{itemize}

The generalization \eqref{eq.KAS} of the Volume Conjecture has been 
numerically verified for a few knots using either state-integral formulas for
Chern-Simons theory when available \cite{DGLZ} or a numerical 
computation of the Kashaev invariant and its numerical asymptotics, lifted to
algebraic numbers \cite{GZ,Ga2}. 

Our goal is to provide an exact, combinatorial definition of the formal
power series $\CZ_M(\hb)$ via formal Gaussian integration using the shape 
parameters and the Neumann-Zagier matrices
of a regular ideal triangulation of $M$. Our definitions

\begin{itemize}
\item
express the putative torsion $\exp(-2S_{M,1})$ and the $n$-loop invariants $S_{M,n}$ manifestly in terms of the shape parameters $z_i$ and the gluing matrices of a regular 
ideal triangulation $\CT$ of $M$;
\item
manifestly deduce that the putative torsion and the $n$-loop invariants 
for $n \geq 2$ are elements of the invariant trace field;
\item
explain the difference of $\CZ_M(\hb)$ for pairs of 
geometrically similar knots studied by Zagier and the second author;
\item
provide an effective way to compute the $n$-loop invariants using 
standard commands of the {\tt SnapPy} toolbox \cite{SnapPy} --- 
as demonstrated for $n=1,2,3$ for hyperbolic knots with at most 14 crossings; 
and
\item
allow efficient tests of the asymptotics of the Volume Conjecture 
\eqref{eq.KAS}, the ``1-loop Conjecture'' \eqref{TRMS1} and
other conjectures in Quantum Topology. 
\end{itemize}

\noindent We note that we only define $\exp(-2S_{M,1})$ up to a sign, and $S_{M,2}$ modulo $\Z/24$. All higher $n$-loop invariants are defined unambiguously.

 Although we give a purely combinatorial definition of $\CZ_M(\hb)$ 
without any knowledge of state integrals or Chern-Simons theory with 
complex gauge group, in Section \ref{sec.SIM} we explain how our definition
follows from the state-integral model of the first author \cite{D1}
and its perturbative expansion.

\subsection{The Neumann-Zagier datum}
\lbl{sec.introNZ}

All manifolds and all ideal triangulations in this paper will be 
oriented. 
The volume of a hyperbolic manifold $M$, appearing in the Volume
Conjecture and contributing to $S_{M,0}$, is already known to have a simple 
expression in terms of shape parameters of a regular ideal 
triangulation, \ie, one that recovers the complete hyperbolic structure of 
$M$. (For extended discussion on regular triangulations, see Section \ref{sec.extensions}.)
If $\CT=\{\Delta_i\}_{i=1}^N$ is a regular ideal triangulation of $M$
with shape parameters 
$z_i \in \BC\setminus\{0,1\}$ for $i=1,\dots,N$, then (\cf\ \cite{DuS, NZ})
\be \mathrm{Vol}(M) 
= \sum_{i=1}^N D(z_i)\,,
\lbl{eq.Dz}
\ee
where $D(z) := \mathrm{Im}\big(\Li_2(z)\big)+\arg(1-z)\log|z|$ is the Bloch-Wigner dilogarithm function.
This formula can also be interpreted as calculating the image of the class 
$[M]:=\sum_i[z_i]\,\in\,\mathcal{B}$ of $M$ in the 
Bloch group $\mathcal{B}$ under the natural map $D:\mathcal{B} \to \BR$. 
An analogous formula, using the class of $M$ in the ``extended'' Bloch group 
$\widehat{\mathcal{B}}$, gives the full complexified volume $S_{M,0}$ 
\cite{Neumann-combi, N,GoZi, Zickert-rep}. 

It is natural to ask whether the class of $M$ in $\mathcal{B}$
determines not only $S_{M,0}$ but the higher $S_{M,n}$ as well. 
This question was posed to the authors several years ago by D. Zagier.
Subsequent computations \cite{GZ,Ga2} indicated that a positive answer
was not possible. For example, there is a family of pairs of pretzel knots 
$\big((-2,3,3+2p),\,(-2,3,3-2p)\big)$ for $p=2,3,...$, as well as the figure-eight knot and its sister, which all have the same class 
in the Bloch group (and classes differing by 6-torsion in the extended
Bloch group), but different invariants $S_{M,n}$ for $n \geq 1$. 

The extra information necessary to determine the $S_{M,n}$ can be described as follows.
Recall that if $\CT$ 
is a regular ideal triangulation of $M$ with $N$ tetrahedra, its shapes $z=(z_1,\dots,z_N)$ satisfy a system of 
polynomial equations, one equation for every 
edge, and one imposing 
parabolic holonomy around the meridian of the cusp \cite{Th, NZ}. Let us set
\be z_i' = (1-z_i)^{-1}\,,\qquad z_i''=1-z_i^{-1}\,.\ee
The equations can then be written in the form
\be \qquad
z^{\mb A} z''{}^{\mb B} :=
\prod_{j=1}^N z_j^{\mb A_{ij}}\big(1-z_j^{-1}\big)^{\mb B_{ij}}=\pm 1\,,
\qquad (i=1,...,N)\,, \lbl{NZintro} \ee
where $\mb A$ and $\mb B$ are $N\times N$ square matrices with integer 
entries, which we call the \emph{Neumann-Zagier matrices} following \cite{NZ}. 

\begin{definition}
\lbl{def.NZdatum}
If $\CT$ is a regular ideal triangulation of $M$, its {\em Neumann-Zagier
datum} (resp., {\em enhanced Neumann-Zagier datum}) is given by the triple
$$
\beta_\CT=(z,\mb A,\mb B) \qquad \text{resp.,}
\qquad\widehat\beta_{\calT}=(z,\mb A,\mb B, f)\,,
$$
where $z$ is a solution to the gluing equations and
$f$ is a combinatorial flattening of $\calT$, a collection of integers that 
we define in Section \ref{sub.flat}.
\end{definition}
As we will discuss in detail in Section \ref{sec.triang},
implicit in the above definition is the dependence of $\beta_{\CT}$ 
and $\widehat\beta_{\calT}$ on the following choices:
\begin{enumerate}
\item
a pair of opposite edges for every oriented ideal tetrahedron
(a so-called choice of {\em quad type}), 
\item
an edge of $\calT$, 
\item
a meridian loop in the boundary of $M$ in general position with respect to $\calT$,
\item a combinatorial flattening.
\end{enumerate}

\subsection{The 1-loop invariant} 
\lbl{sub.1loopdef}

\begin{definition}
\lbl{def.1loopdef}
Given a one-cusped hyperbolic manifold $M$ with regular ideal 
triangulation $\CT$ and enhanced Neumann-Zagier datum $\widehat\beta_\CT$ 
we define:
\be 
\qquad\tau_{\CT} := 
\pm \frac12\det\big(\mb A\Delta_{z''}
+\mb B\Delta_z^{-1}\big)z^{f''}z''^{-f}
\quad\in\,E_M/\{\pm 1\}\,,
\lbl{torintro} 
\ee
where $\Delta_z:=\diag(z_1,...,z_N)$ and 
$\Delta_{z''}:=\diag(z_1'',...,z_N'')$ are diagonal matrices, and  
$z^{f''}z''^{-f} := \prod_i z_i{}^{f''_i}z''_i{}^{-f_i}$. 
\end{definition}

\noindent Note that $\tau_{\CT}$ takes value in the invariant trace field of $M$
and is only defined up to a sign. In Section \ref{sec.torsion} we will show that

\begin{theorem}
\lbl{thm.indep}
$\tau_{\CT}$ is independent of the quad type of $\calT$, the chosen
edge of $\calT$, the choice of a meridian loop, and the choice of 
a combinatorial flattening.
\end{theorem}

We now consider the dependence of $\tau_{\CT}$ on the choice of a regular
ideal triangulation of $M$. It is well known that the set
$\calX$ of ideal triangulations of a cusped 
hyperbolic manifold is non-empty \cite{Ca} and connected by 2--3 moves 
\cite{Ma1,Ma2,Pi}. That is, a sequence of 2--3 moves can be used to take 
any one ideal triangulation to any other. The subset $\calX_{\rho_0}$
of $\calX$ of regular triangulations is also 
non-empty, see Section \ref{sec.extensions}. Topologically, these are the 
triangulations without any univalent edges 
\cite{Champanerkar, BDR-V, DnG, Till}. 
We will prove in Section \ref{sec.torsion} that

\begin{theorem}
\lbl{thm.1}
$\tau_{\CT}$ is constant on every component of $\calX_{\rho_0}$ connected by 2--3 moves.
\end{theorem}

\subsection{Expectations}
\lbl{sub.ought}

We may pose some questions and conjectures about the 1-loop
invariant $\tau_{\CT}$ and the structure of the set $\calX_{\rho_0}$.
Let us begin with two questions whose answer is unfortunately
unknown.

\begin{question}
\lbl{que.can}
Is $\calX_{\rho_0}$ connected by 2--3 moves? 
\end{question}

\begin{question}
\lbl{que.constanttau}
Is $\tau_{\CT}$ constant on the set $\calX_{\rho_0}$?
\end{question}

\noindent Clearly, a positive answer to the first question implies a positive answer to the second.

Despite the unknown answer to the above questions, with additional effort we
can still define a distinguished component of $\calX_{\rho_0}$,
and thus obtain a topological invariant of $M$. Namely, let $\calX^{\EP}_M
\subset \calX_{\rho_0}$ denote the subset that consists of regular 
refinements of the canonical (Epstein-Penner) ideal cell decomposition of 
$M$ \cite{EP}. $\calX^{\EP}_M$ is canonically associated to a cusped
hyperbolic manifold $M$.
A detailed description of $\calX^{\EP}_M$ is given in \cite[Sec.6]{GHRS}.
$\calX^{\EP}_M$ is a finite set that generically consists
of a single element. In Section \ref{sub.EP} we will show the following.

\begin{proposition}
\lbl{prop.EP}
$\calX^{\EP}_M$ lies in a connected component of $\calX_{\rho_0}$.
Consequently, the value of $\tau_{\CT}$ on $\calX^{\EP}_M$ is a
topological invariant $\tau_M$ of $M$.
\end{proposition}

\noindent Admittedly, it would be more natural to show that $\tau_\CT$ is constant on all of $\calX_{\rho_0}$.
Proposition \ref{prop.EP} appears 
to be an artificial way to construct a much needed topological invariant
of cusped hyperbolic 3-manifolds.

Our next conjecture compares our torsion $\tau_M$ with the 
nonabelian Reidemeister torsion $\tau^{\Reid}_M$ of $M$ with respect
to the meridian defined in \cite{Po,Db}. 
\begin{conjecture}
\lbl{conj.1loop}
For all hyperbolic knot complements we have
$\tau^{\Reid}_M=\pm \tau_{M}$.
\end{conjecture}
\noindent 
Numerical evidence for the above conjecture is presented in
Appendix \ref{app.compute} using Dunfield's computation of $\tau^{\Reid}_M$
via {\tt SnapPy} \cite{DFJ}. Observe that both sides of the equation in
Conjecture \ref{conj.1loop} are algebraic numbers (defined up to a sign) 
that are elements of the invariant trace field of $M$. Moreover, if $M$ has
a regular ideal triangulation with $N$ ideal tetrahedra and its 
fundamental group is generated with $r$ elements, then $\tau_M$ and
$\tau^{\Reid}_M$ are essentially given by the determinant of square matrices
of size $N$ and $3 r-3$, respectively. It is still unclear to us
how to relate these two matrices or their determinants.

By definition, $\tau^{\Reid}_M \in E_M^*$. Thus, a mild but important 
corollary of Conjecture \ref{conj.1loop} is that $\tau_M$ is nonzero. 
This is a crucial ingredient, necessary for the definition of the higher
loop invariants $S_{M,n}$ using perturbation theory.

\subsection{The higher-loop invariants}
\lbl{sub.nloop}

In this section we define the higher loop invariants $S_{\CT,n}$ for $n \geq 2$. They are analyzed in detail in Section \ref{sec.SIM}, using a state integral \eqref{SIM}. The result, however, may be summarized as follows.
Let us introduce a formal power series 
\be
\lbl{eq.psixz}
\psi_{\hb}(x;z)=\exp\left(
\sum_{n,\, k,\, 2n+k-2 > 0} 
\frac{\hb^{n+\tfrac k2-1} (-x)^kB_n}{n!k!}\, \Li_{2-n-k}(z^{-1})
\right) \in \BQ(z)[\![x,\hbar^{\frac12}]\!]
\ee
where $B_n$ is the $n^{\rm th}$ {\em Bernoulli number} (with $B_1=+1/2$),
and $L_m(z)$ is the $m^{\rm th}$ polylogarithm. Note that 
$\Li_m(z) \in (1-z)^{-m-1}\BZ[z]$ is a rational function for all nonpositive integers $m$. This formal series comes from the asymptotic expansion of the 
quantum dilogarithm function after removal of its two leading asymptotic
terms \cite{Ba, FK-QDL, Fa}. The quantum dilogarithm is 
the Chern-Simons partition function of a single tetrahedron and its asymptotics
are studied in detail in Section \ref{sec.SIM}.

We fix an enhanced Neumann-Zagier datum $\widehat 
\beta_{\calT}=(z,\mb A, \mb B,f)$
of an oriented 1-cusped manifold $M$ and a regular ideal triangulation
$\calT$ with $N$ tetrahedra. Let $\eta=\mb A f + \mb B f''$. We assume that 
\be
\det(\mb B) \neq 0\,,\qquad  \tau_M \neq 0\,. \notag
\ee
The condition  $\det(\mb B) \neq 0$ is always satisfied with a suitable 
labeling of shapes; see Lemma \ref{lemma.B}. In that case,
Lemma \ref{lemma.BA} implies that 
\be
\CH=-\mb B^{-1}\mb A+\Delta_{z'}\,,
\ee
is a symmetric matrix, where 
$\Delta_{z'}=\diag(z_1',...,z_N')$. 
We define
\be
f_{\calT,\hb}(x;z)=
\exp \bigg(-\frac{\hbar^{\frac12}}{2} x^T \mb B^{-1} \eta +\frac\hbar8 f^T\mb B^{-1}\mb A f \bigg) 
\prod_{i=1}^N \psi_{\hb}(x_i,z_i) \in \BQ(z)[\![x,\hbar^{\frac12}]\!]\,,
\quad
\ee
where $x=(x_1,...,x_N)^T$ and $z=(z_1,...,z_N)$.
Assuming that $\CH$ is invertible,
a formal power series $f_\hbar(x)\in \BQ(z)[\![x,\hbar^{\frac12}]\!]$ has a {\em formal Gaussian integration}, given by (\cf\ \cite{BIZ})
\be \label{formalG}
\la f_{\hb}(x)\ra = \frac{\int dx \, e^{-\frac12x^T \CH \,x}f_{\hb}(x)}{\int dx \, e^{-\frac12x^T \CH \,x}}\,.
\ee
This integration is defined by expanding $f_\hbar(x)$ as a series in $x$, and then formally integrating each monomial, using the quadratic form $\CH^{-1}$ to contract $x$-indices pairwise. 

\begin{definition}
\lbl{def.Sn}
With the above conventions, we define
\be
\exp\left(\sum_{n=2}^\infty S_{\calT,n}(z) \hb^{n-1}\right)
:=
\la f_{\calT,\hb}(x;z)
\ra \,.
\ee
\end{definition}

\begin{remark}
\lbl{rem.tau3n}
Notice that the result involves only integral powers of $\hbar$
and each term is a rational function in the complex numbers $z$.
Moreover, $S_{\calT,n} \in \tau_{\CT}^{-3n+3} \BQ[z,z',z'']$ 
for all $n \geq 2$. This follows from the fact that the connected Feynman
diagrams that contribute to $S_{\CT,n}$ have at most $3n-3$ edges and each
edge (contracted by $\CH^{-1}$) contributes a factor of $\det(\CH)^{-1}$. 
Thus, we can also write
\be \label{tordivide}
\exp\left(\sum_{n=2}^\infty S_{\calT,n} \hb^{n-1}\right)=
1+ \sum_{n=1}^\infty  \frac{\wt S_{\calT,n}}{\tau_{\CT}^{3n}}\hb^n \,, 
\ee
where $\wt S_{\CT,n} \in \BQ[z,z',z'']$. Experimentally, it appears that
$\wt S_{\calT,n}$ have lower complexity than $S_{\calT,n}$, see Appendix 
\ref{app.compute}.
\end{remark}

\subsection{Feynman diagrams}
\lbl{sub.feynman2}

A convenient way to organize the above definition is via Feynman diagrams, 
using Wick's theorem to express each term $S_{\CT,n}$ as a finite sum of 
connected diagrams with at most $n$ loops, where the number of loops of
a connected graph is its first betti number. This is well known and 
explained in detail, \eg\ in \cite[Ch.9]{MirrorSym}, \cite{BIZ,Polyak}.

The Feynman rules for computing the $S_{\CT,n}$, described in Section \ref{sec.SIM}, turn out to be the following.%
\footnote{To derive these from \eqref{formalG}, one should first rescale $x\to \hbar^{-\frac12}x$.} %
One draws all connected graphs $D$ with vertices of 
all valencies, such that
\be 
\lbl{eq.LD} 
L(D) := (\text{\# 1-vertices + \# 2-vertices + \# loops}) \;\leq n\,. 
\ee
In each diagram, the edges represent an $N\times N$ propagator
\be \text{propagator}\; :\quad \Pi=\hbar\,\CH^{-1}\,,
\ee
while each $k$-vertex comes with an $N$-vector of factors $\Gamma^{(k)}_i$,
\be 
\label{vertexintro}  
\Gamma^{(k)}_i = (-1)^k\sum_{p\,
=\,\alpha_k}^{\alpha_k +n-L(D)}\frac{\hbar^{p-1}B_p}{p!}\Li_{2-p-k}(z_i^{-1}) 
+ \begin{cases} -\tfrac12(\mb B^{-1}\eta)_i & k= 1 \\ \;\;0 & k \geq 2 
\end{cases}\,,
\ee
where $\alpha_k = 1$ (resp., $0$) if $k=1,2$ (resp., $k\geq 3$).
The diagram $D$ is then evaluated by contracting the \emph{vertex factors} $\Gamma^{(k)}_i$ with 
propagators, multiplying by a standard {\em symmetry factor}, and taking the 
$\hbar^{n-1}$ part of the answer. In the end, $S_{M,n}$ is the sum of 
evaluated diagrams, plus an additional {\em vacuum} contribution
\be 
\label{vacintro} 
\Gamma^{(0)} = \frac{B_n}{n!}\sum_{i=1}^N\Li_{2-n}(z_i^{-1}) 
+ \begin{cases} \tfrac18f\cdot\mb B^{-1}\mb Af & n=2 \\ 0 & n\geq 3 
\end{cases}\,.
\ee
To illustrate the above algorithm, we give the explicit formulas
for $S_2$ and $S_3$ below.

\subsection{The $2$-loop invariant}
\lbl{sub.2loop}

The six diagrams that contribute to $S_{M,2}$ are shown in Figure 
\ref{fig.diags2intro}, together with their symmetry factors.

\begin{figure}[htb]
\centering
\includegraphics[width=4in]{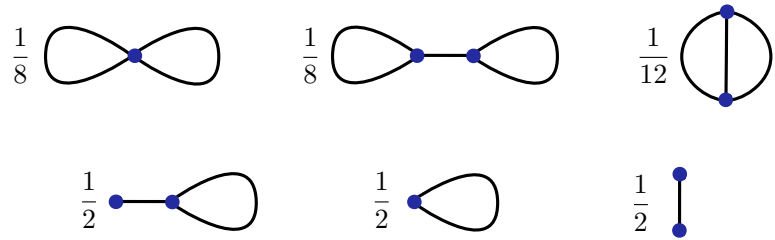}
\caption{Diagrams contributing to $S_{M,2}$ with symmetry factors. The top 
row of diagrams have exactly two loops, while the bottom row have fewer 
loops and additional $1$-vertices and $2$-vertices.}
\label{fig.diags2intro}
\end{figure}

\noindent Their evaluation gives the following formula for $S_{\CT,2}$:
\begin{align} \label{2loopexplicit}
S_{\CT,2} &= \coeff\left[ 
\frac18  \Gamma^{(4)}_i (\Pi_{ii})^2 
+ \frac18\Pi_{ii}\Gamma^{(3)}_i\Pi_{ij}\Gamma^{(3)}_j\Pi_{jj} 
+ \frac{1}{12}\Gamma^{(3)}_i(\Pi_{ij})^3\Gamma^{(3)}_j \right. \\
&\hspace{1in}
\left. + \frac12 \Gamma^{(1)}_i\Pi_{ij}\Gamma^{(3)}_j\Pi_{jj} 
+ \frac12 \Gamma^{(2)}_i\Pi_{ii} 
+ \frac12\Gamma^{(1)}_i\Pi_{ij}\Gamma^{(1)}_j ,\; \hbar\right]+\Gamma^{(0)}\,, \notag
\end{align}
where all the indices $i$ and $j$ are implicitly summed from $1$ to $N$ 
and $\coeff[f(\hb),\,\hb]$ denotes the coefficient of $\hb$ of a power series
$f(\hb)$. Concretely, the 2-loop contribution from the vacuum energy is 
$\Gamma^{(0)}=\frac{1}{8}f^T \mb B^{-1}\mb A f-\frac1{12}\sum_iz_i'$, and
the four vertices that appear only contribute at leading order,
\be 
\Gamma^{(1)}_i = \frac{z_i'-(\mb B^{-1}\eta)_i}{2}\,,
\qquad \Gamma^{(2)}_i = \frac{z_iz_i'^2}{2}\,,
\qquad \Gamma^{(3)}_i = -\frac{z_iz_i'^2}{\hbar}\,,
\qquad \Gamma^{(4)}_i = -\frac{z_i(1+z_i)z_i'^3}{\hbar}\,.
\ee
We expect $S_{\CT,2}$ to be well defined modulo $\Z/24$, and this is exactly
what happens in hundreds of examples that we computed.

\subsection{The $3$-loop invariant}
\lbl{sub.3loop}

For the next invariant $S_{\CT,3}$, all the diagrams of Figure 
\ref{fig.diags2intro} contribute, collecting the coefficient
of $\hb^2$ of their evaluation. In addition, there are $34$ new diagrams 
that satisfy 
the inequality \eqref{eq.LD}; they are shown in Figures \ref{fig.diags3} 
and \ref{fig.diags3dot}. Calculations indicate that the 3-loop invariant 
$S_{\CT,3}$ is well defined, independent of the regular triangulation $\CT$.
The invariants $S_{\CT,0},\tau_{\CT},S_{\CT,2},S_{\CT,3}$ have been programmed in 
{\tt Mathematica} and take as input a Neumann-Zagier datum available
from {\tt SnapPy} \cite{SnapPy}.

For the 4-loop invariant, there are $291$ new diagrams. A {\tt python}
implementation will be provided in the future. For large $n$, one expects about
$n!^2 C^n$ diagrams to contribute to $S_n$. 

\begin{remark}
\lbl{rem.noflattening}
Note that the $n$-loop invariant for $n \geq 3$ is independent of the 
combinatorial flattening and in fact depends only on 
$(\mb B^{-1} \mb A, \mb B^{-1} \eta,z)$.
\end{remark}

\begin{figure}[htpb]
\centering
\includegraphics[width=6.5in]{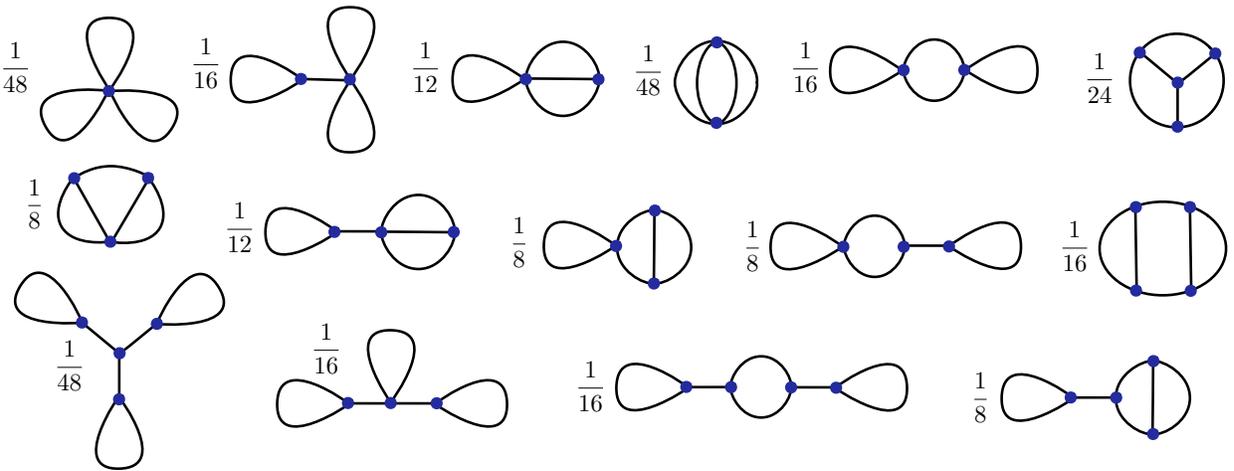}
\caption{Diagrams with three loops contributing to $S_3$.}
\lbl{fig.diags3}
\end{figure}

\begin{figure}[htpb]
\centering
\includegraphics[width=6.5in]{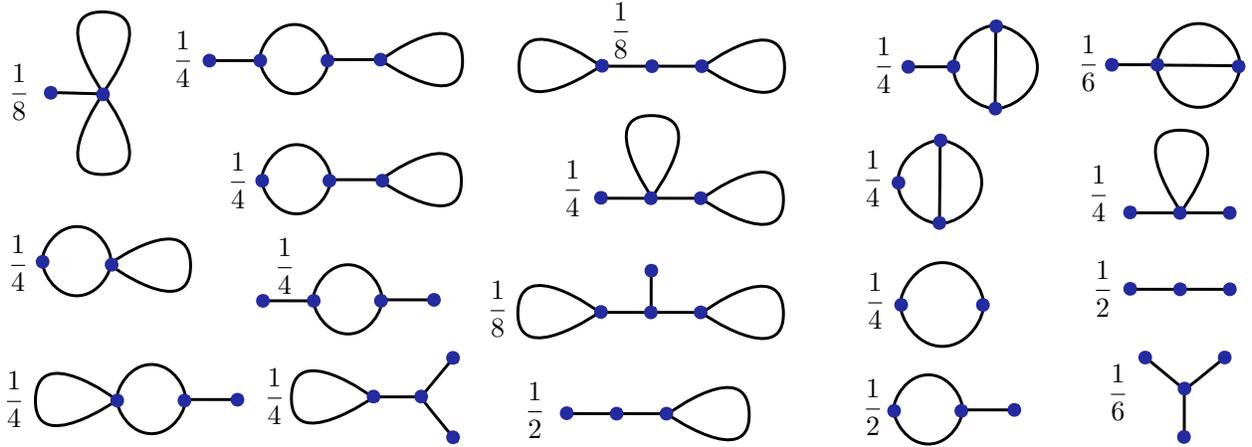}
\caption{Diagrams with $1$-vertices and $2$-vertices 
contributing to $S_3$.}
\lbl{fig.diags3dot}
\end{figure}

\subsection{Generalizations}
\lbl{sec.genintro}

There are several natural extensions of the results presented above.
First, one could attempt to prove the independence of the all-loop invariants $Z_{\CT}(\hb)$, including the entire series of $S_{\CT,n}$'s, under 2--3 moves and different choices of Neumann-Zagier datum. This was done non-rigorously in \cite{D1}, but a full mathematical argument in the spirit of Theorems \ref{thm.indep} and \ref{thm.1} is still missing. We hope to address this in future work.

In a different direction, one can extend the formulas for $\tau_\CT$ and $S_{\CT,n}$ to
\begin{itemize}
\item[--] manifolds with multiple cusps,
\item[--] representations other than the discrete faithful,
\item[--] representations with non-parabolic meridian holonomy,
\item[--] non-hyperbolic manifolds.
\end{itemize}
The only truly necessary condition is that a 3-manifold $M$ have a topological ideal triangulation $\CT$ that --- upon solving gluing equations and using a developing map --- reproduces some desired representation $\rho:\pi_1(M)\to \PSL(2,\C)$. We call such an ideal triangulation $\rho$-regular, and in Section \ref{sec.extensions} we will briefly discuss most of the above generalizations. In particular, we will demonstrate in Sections \ref{sec.utor41} and \ref{sec.uSIM} how to extend $\tau_\CT,\,S_{\CT,n}$ to rational functions on the character variety of a (topologically) cusped manifold. The generalization to multiple cusps is also quite straightforward, but left out mainly for simplicity of exposition.


\subsection*{Acknowledgment}
The authors wish to thank Nathan Dunfield, Walter Neumann, Jose\-phine Yu
and Don Zagier for many extremely enlightening conversations.


\section{Mechanics of triangulations}
\lbl{sec.triang}

We begin by reviewing the gluing rules for ideal hyperbolic tetrahedra and 
the equations that determine their shape parameters. We essentially follow 
the classic \cite{Th, NZ}, but find it helpful to work with 
additive logarithmic (rather than multiplicative) forms of the gluing  equations.
Recall that all manifolds and all ideal triangulations are {\em oriented}.

\subsection{Ideal tetrahedra}
\lbl{sec.tet}

Combinatorially, an \emph{oriented ideal tetrahedron} $\Delta$ is a topological ideal tetrahedron with three complex \emph{shape parameters} $(z,z',z'')$ assigned to pairs of opposite edges (Figure \ref{fig.tet}). The shapes always appear in the same cyclic order (determined by the orientation) around every vertex, and they satisfy

\begin{subequations} \lbl{tetexp}
\be zz'z'' = -1\,, \hspace{0in}\lbl{tetexpP} \ee
and
\be z''+z^{-1}-1=0\,. \hspace{0in}\lbl{tetLag} \ee
\end{subequations}
In other words, $z'=1/(1-z)$ and $z''=1-z^{-1}$. We call the tetrahedron \emph{non-degenerate} if none of the shapes take values in $\{0,1,\infty\}$, \ie, $z,z',z''\in \C^*\backslash\{1\}$. It is sufficient to impose this on a single one 
of the shapes.

\begin{figure}[htb]
\includegraphics[width=2in]{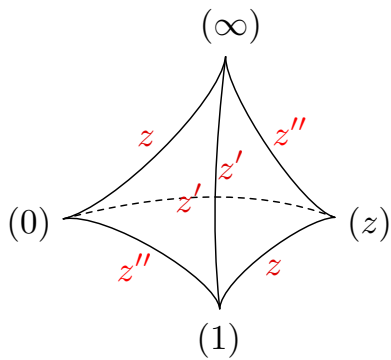}
\caption{An ideal tetrahedron
}
\lbl{fig.tet}
\end{figure}

Borrowing common terminology from the theory of normal surfaces, \cf\ 
\cite{Bu,Kang,KR,Till2}, 
we define the \emph{quadrilateral type} (in short, \emph{quad type}) of $\Delta$ to be the distinguished pair of opposite edges labelled by $z$. Clearly, there is a three-fold choice of quad type for any oriented ideal tetrahedron. Different choices correspond to a cyclic permutation of the vector $(z,z,',z'')$,
which leaves relations \eqref{tetexp} invariant.

Geometrically, the shape parameters determine a $\PSL(2,\C)$ structure on $\Delta$. Equivalently, they determine a hyperbolic structure, possibly of negative volume. We can then describe the ideal hyperbolic tetrahedron $\Delta$ as the convex hull of four ideal points in hyperbolic three-space $\mathbb H^3$, whose cross-ratio is $z$ (or $z'$, or $z''$). 
Each shape $z$ fixes the complexified dihedral angle on 
the edge it labels, via
\be z = \exp(\text{torsion \!+\! $i$\,angle})\,,
\ee
and similarly for  $z'$, $z''$.
%
%


\subsection{The gluing matrices}
\lbl{sec.glue}

We now discuss an important combinatorial invariant of ideal triangulations,
namely the gluing and Neumann-Zagier matrices, their symplectic
properties, and the notion of a combinatorial flattening. Although these
notions are motivated by hyperbolic geometry (namely the gluing of ideal
tetrahedra around their faces and edges to describe a complete hyperbolic
structure on a cusped manifold), we stress that these notions make sense
for arbitrary 3-manifolds with torus boundary, and for triangulations 
whose gluing equations may not have solutions in $\BC\setminus\{0,1\}$.

Let $M$ be an oriented one-cusped manifold with an ideal 
triangulation $\calT=\{\Delta_i\}_{i=1}^N$ and a choice of 
quad type. 

The choice of quad, combined with the orientation of $\CT$ and $M$ allow us
to attach variables $(Z_i,Z_i',Z_i'')$ to each tetrahedron $\Delta_i$.
An Euler characteristic argument shows that the triangulation has 
$N$ edges $E_I$, $I=1,...,N$. For each edge $E_I$ we introduce a gluing 
equation of the form
\be E_I:\quad \sum_{i=1}^N \big(\mb{G}_{Ii}Z_i+\mb{G}_{Ii}'Z_i'
+\mb{G}_{Ii}''Z_i''\big) = 2\pi i\,\qquad I=1,...,N\,, \lbl{glue} \ee
where $\mb{G}_{Ii} \in \{0,1,2\}$ (resp., $\mb{G}_{Ii}'$, $\mb{G}_{Ii}''$) is 
the number of times an edge of tetrahedron $\Delta_i$ with parameter $Z_i$ 
(resp., $Z_i'$, $Z_i''$) is incident to the edge $E_I$ in the triangulation. 
In addition, we impose the equations
\be 
Z_i + Z_i'+Z_i''=i\pi \label{tetP}\,.
\ee
for $i=1,\dots,N$. The Equations \eqref{glue} are not all independent. 
For a one-cusped 
manifold, every edge begins and ends at the cusp, which implies 
$\sum_{I=1}^N \mb G_{Ii}=\sum_{I=1}^N \mb G_{Ii}'=\sum_{I=1}^N \mb G_{Ii}''=2$, and 
therefore that the sum of the left-hand sides of Equations \eqref{glue} 
equals $2\pi i N$. This is the only linear dependence in case of
one cusp. In general, there is one relation per cusp of $M$, 
as follows from \cite[Thm.4.1]{Neumann-combi}.

An oriented peripheral simple closed curve $\mu$ (such as a meridian) on the 
boundary of $M$, 
in general position with the triangulation of the boundary torus that
comes from $\CT$, also gives rise to a gluing equation. We assume that the curve is simple (has no self intersections), and set the signed sum of edge parameters on the dihedral angles subtended by the curve to zero. A parameter is counted with a plus sign (resp. minus sign) if an angle is subtended in a counterclockwise (resp. clockwise) direction as viewed from the boundary. These rules are the same as described in \cite{Neumann-combi}, and demonstrated in Section \ref{sec.41triang} below.

Let us choose such a peripheral curve $\mu$. We choose a meridian if $M$ is a knot complement. The gluing equation associated to $\mu$ then takes the form
\be \mu:\quad  \sum_{i=1}^N \big( \mb{G}_{N+1,i} Z_i+\mb{G}_{N+1,i}' Z_i'
+\mb{G}_{N+1,i}''Z_i'' \big) = 0\,,  \lbl{merid} \ee
with $\mb{G}_{N+1,i},\,\mb{G}_{N+1,i}',\,\mb{G}_{N+1,i}''\in \Z$. 

\subsection{The Neumann-Zagier matrices}
\lbl{sec.NZ}

The matrices $\mb G$, $\mb G'$ and $\mb G''$ 
have both symmetry and  
redundancy. We have already observed that any one of the edge 
constraints \eqref{glue} can be removed. Let us then ignore the  
edge $I=N$. We can also use \eqref{tetP} to eliminate one of the three 
shapes for each tetrahedron. We choose this canonically to be $Z_i'$,
though which pair of edges is labelled $Z_i'$ depends on the choice of 
quad type for the tetrahedron. Then the first $N-1$ edge equations 
and the meridian ($\mu$) equation are equivalent to
\be 
\sum_{j=1}^N \big(\mb A_{ij}Z_j+\mb B_{ij}Z_j''\big) = i\pi\, \eta_i\,,
\qquad i=1,...,N\,, 
\lbl{NZlin} 
\ee
where 
\be
 \mb A_{ij} = \left\{\begin{array}{c@{\quad}l}
  \mb G_{ij}-\mb G_{ij}' & I\neq N \\[.2cm]
    \mb G_{N+1,j}-\mb G_{N+1,j}' & i= N
  \end{array}\right. \qquad
  \mb B_{ij} = \left\{\begin{array}{c@{\quad}l}
  \mb G_{ij}''-\mb G_{ij}' & I\neq N \\[.2cm]
    \mb G_{N+1,j}''-\mb G_{N+1,j}' & i= N\,;
  \end{array}\right. \qquad
\ee
and
\be \eta_i := \left\{\begin{array}{c@{\quad}l}
  2-\sum_{j=1}^N \mb G_{ij}' & i\neq N \\[.2cm]
    -\sum_{j=1}^N \mb G_{N+1,j}' & i= N\,.
  \end{array}\right.
\lbl{defn}
\ee
We will generally assume $Z$, $Z''$, and $\eta$ to be column vectors, 
and write %
$ \mb A\, Z+\mb B \,Z''=i\pi \eta\,.$
The matrices $(\mb G,\mb G',\mb G'')$ as well as
$(\mb A, \mb B,\eta)$ can easily be obtained from \texttt{SnapPy} 
\cite{SnapPy}, as is illustrated in Appendix \ref{app.compute}.

The Neumann-Zagier matrices $\mb A$ and $\mb B$ have a remarkable property: 
they are the top 
two blocks of a $2N\times 2N$ symplectic matrix \cite{NZ}. It follows that
\be \mb A\mb B^T = \mb B\mb A^T \,, 
\lbl{NZsymp} \ee
and that the $N\times 2N$ block $(\mb A\;\mb B)$ has full rank. This 
symplectic property is crucial for defining the state integral of \cite{D1},
for defining our formal power series invariant $\CZ_M(\hbar)$, and for the 
combinatorial proofs of topological invariance of the $1$-loop invariant.
A detailed discussion of the symplectic properties of the Neumann-Zagier
matrices $\mb A, \mb B$ is given in Appendix \ref{app.symp}.

\subsection{Combinatorial flattenings}
\lbl{sub.flat}

We now have all ingredients to define what is a combinatorial flattening.

\begin{definition}
\lbl{def.flattening}
Given an ideal triangulation $\calT$ of $M$, a {\em combinatorial flattening} 
is a collection of $3N$ integers $(f_i,f_i',f_i'') \in \BZ^3$ for $i=1,\dots,N$
that satisfy
\begin{subequations}
\be f_i+f_i'+f_i'' = 1\,,\qquad i=1,...,N\,, \lbl{flattrip} \ee
\be \label{flatG} 
\sum_{i=1}^N \big(\mb{G}_{Ii}f_i+\mb{G}_{Ii}'f_i'+\mb{G}_{Ii}''f_i''\big) 
= \begin{cases} 2 & I=1,...,N \\
0 & I=N+1 \end{cases} \,.
\ee
\end{subequations}
\end{definition}
Note that if we eliminate $f'$ using Equation \eqref{flattrip}, a flattening
is a pair of vectors $(f,f'') \in \BZ^{2N}$ that satisfies
\be 
\label{flatAB} \mb A f+\mb B f'' = \eta\,. 
\ee
Evidently, Equation \eqref{flatAB} is a system of linear Diophantine equations.
Neumann proved in \cite[Lem.6.1]{Neumann-combi} that every ideal triangulation
$\CT$ has a flattening.

\begin{remark}
\lbl{rem.Nflattening}
Combinatorial flattenings should not be confused with the (geometric)
flattenings of \cite[Defn.3.1]{N}. The latter flattenings are
coherent choices of logarithms for the shape parameters $z,z',z''$ of a 
complex solution to the gluing equations. On the other hand, our 
combinatorial flattenings are independent of a solution to the gluing 
equations. In the rest of the paper, the term {\em flattening} will mean
a {\em combinatorial flattening} in the sense of Definition 
\ref{def.flattening}.
\end{remark}

\subsection{The shape solutions to the gluing equations}
\lbl{sub.shapesol}

If we exponentiate the equations \eqref{tetP}, and set 
$(z_i,z_i',z_i'')=(e^{Z_i}, e^{Z_i'},e^{Z_i''})$, we obtain that 
$(z_i,z_i',z_i'')$ satisfy Equation \eqref{tetexp}. If we combine the
exponentiated equations \eqref{NZlin} with the 
nonlinear relation \eqref{tetLag} for each tetrahedron, we obtain
\be z^{\mb A}z''^{\mb B} = z^{\mb A}(1-z^{-1})^{\mb B} = (-1)^\eta\,, 
\lbl{glueeqs} \ee
where $z^{\mb A} := \prod_j z_j^{\mb A_{ij}}$. These $N$ equations in $N$ 
variables are just the gluing equations of Thurston \cite{Th} and Neumann 
and Zagier \cite{NZ}, and fully capture the constraints imposed by the 
gluing. For hyperbolic 
$M$, a triangulation $\CT$ is regular precisely when one of the solutions to 
\eqref{glueeqs} corresponds to the complete hyperbolic structure.

\subsection{Example: $\mb{4_1}$}
\lbl{sec.41triang}

As an example, we describe the enhanced Neumann-Zagier datum of 
the figure-eight knot complement $M$. It has a well known regular 
ideal triangulation $\CT$ consisting of $N=2$ tetrahedra, to which we assign logarithmic shape parameters $(Z,Z',Z'')$ and $(W,W',W'')$.

\begin{figure}[htb]
\centering
\includegraphics[width=4.5in]{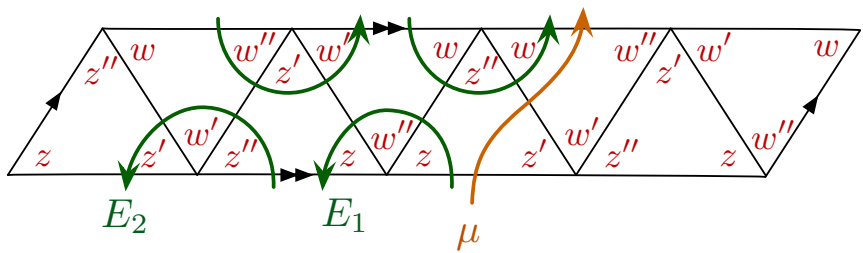}
\caption{The boundary of the cusp neighborhood for the figure-eight knot.}
\lbl{fig.41cusp}
\end{figure}

A map of the boundary of the cusp neighborhood is shown in Figure 
\ref{fig.41cusp}. We have chosen one of $3^2$ possible cyclic labelings by 
$Z$'s and $W$'s (\ie\ one of $3^2$ possible quad types). Each of the edges intersects the cusp twice, so it 
is easy to read off from Figure \ref{fig.41cusp} that the edge 
constraints \eqref{glue} are
\begin{align*} E_1:&\quad 2Z+ Z''+2W+W'' = 2\pi i \\
 E_2:&\quad 2Z'+Z''+2W'+W'' = 2\pi i\,.
\end{align*}
The sum of the left-hand sides is automatically $4\pi i$, so we can choose to ignore the
second constraint. If we 
choose the meridian path $\mu$ as in Figure \ref{fig.41cusp}, the meridian
constraint \eqref{merid} is
\be \mu:\quad -Z'+W = 0 \notag 
\ee
Putting together the first edge constraint and the meridian into 
matrices, we have
\be \begin{pmatrix} 2 & 2 \\0 & 1 \end{pmatrix}\begin{pmatrix} Z
\\W\end{pmatrix}
 + \begin{pmatrix} 0 & 0 \\ -1 & 0 \end{pmatrix}\begin{pmatrix} Z'
\\W'\end{pmatrix}
 + \begin{pmatrix} 1 & 1 \\0 & 0 \end{pmatrix}\begin{pmatrix} Z''
\\W''\end{pmatrix} = i\pi \begin{pmatrix} 2 \\ 0 \end{pmatrix} \,.
\notag \ee
Using $Z+Z'+Z''=W+W'+W''=i\pi$ to eliminate $Z'$ and $W'$, we get
\be \begin{pmatrix} 2 & 2 \\1 & 1 \end{pmatrix}\begin{pmatrix} Z
\\W\end{pmatrix}
 + \begin{pmatrix} 1 & 1 \\ 1 & 0 \end{pmatrix}\begin{pmatrix} Z''
\\W''\end{pmatrix} = i\pi \begin{pmatrix} 2 \\ 1 \end{pmatrix}\,.
\notag \ee
From this last expression, we can read off 
\be 
\lbl{NZ41}
\mb A = \begin{pmatrix} 2 & 2 \\ 1 &  1 \end{pmatrix}\,,\qquad 
\mb B = \begin{pmatrix} 1 & 1 \\ 1 & 0 \end{pmatrix}\,,\qquad \eta 
= \begin{pmatrix} 2 \\  1 \end{pmatrix}\,.
\ee
The two gluing equations \eqref{glueeqs} are then
\be \label{41glue}
z^2w^2z''w''=1\,, \qquad zwz'' =-1 \,,
\ee
with $z''=1-z^{-1}$ and $w''=1-w^{-1}$. The solution for the complete hyperbolic 
structure is $z=w=e^{i\pi/3}$.

Finally, a flattening $(f_z,f_z',f_z'';f_w,f_w',f_w'')\in \Z^{6}$, \ie\ an 
integer solution to $\mb A f+\mb B f'' = \eta$ and $f+f'+f''=1$, is given by
\be \lbl{flat41} (f_z,f_z',f_z'';f_w,f_w',f_w'') = (0,1,0;1,0,0) \ee
It is easy to see that every flattening has the form
$(a, b, 1-a-b; b, a, 1 - a - b)$ for integers $a, b$.


\section{Topological invariance of our torsion}
\lbl{sec.torsion}

Given a one-cusped hyperbolic manifold $M$ with regular triangulation $\CT=\{\Delta_i\}_{i=1}^N$ and Neumann-Zagier datum $\widehat\beta_\CT=(z,\mb A,\mb B,f)$, we have proposed that the nonabelian torsion is given by
\be \label{tor}
\qquad \tau_{\CT} := 
\pm \frac12\det\big(\mb A\Delta_{z''}
+\mb B\Delta_z^{-1}\big)z^{f''}z''^{-f} 
\quad\in\,E_M /\{\pm 1\}\,,
\ee
where $\Delta_{z}=\diag(z_1,...,z_N)$, and similarly for $\Delta_{z''}$. Since $(z,z',z'') \in E_M$ we must have $\tau_\CT \in E_M$ as well.

After a brief example of how the formula \eqref{tor} works, we will proceed to prove Theorems \ref{thm.indep} and \ref{thm.1} on the topological invariance of $\tau_\CT$.
We saw in Section \ref{sec.triang} that the Neumann-Zagier datum depends not only on a triangulation $\CT$, but also on a choice of
\begin{enumerate}
\item quad type for $\CT$,
\item one edge of $\CT$ whose gluing equation is redundant,
\item normal meridian path,
\item flattening $f$.
\end{enumerate}
We will begin by showing that $\tau_\CT$ is independent of these four choices, and then show that it is invariant under 2--3 moves, so long as the 2--3 moves connect two regular triangulations.

The four choices here are independent, and can be studied in any order. However, in order to prove independence of flattening, it is convenient to use a quad type for which the matrix $\mb B$ is non-degenerate. Such a quad type can always be found (Lemma \ref{lemma.B}), but is not automatic. Therefore, we will first show invariance under change of quad type, and then proceed to the other choices. It is interesting to note that of all the arguments that follow (including the 2--3 move), independence of flattening is the only one that requires the use of the full gluing equations $z^{\mb A}z''{}^{\mb B}=(-1)^\eta$.

\subsection{Example: $\mb{4_1}$ continued}
\lbl{sub.torexample}

To illustrate the Definition \ref{torintro}, consider the figure-eight knot 
complement again. From Section \ref{sec.41triang}, we already have one 
possible choice for the Neumann-Zagier matrices \eqref{NZ41} and a 
generic flattening \eqref{flat41}. We use them to obtain
\begin{align}
\pm\tau_{\mb{4_1}} &= \frac12 \det\left[\begin{pmatrix} 2 & 2 \\ 1 & 1 
\end{pmatrix}\begin{pmatrix} z'' & 0 \\ 0 & w'' \end{pmatrix} 
+ \begin{pmatrix} 1 & 1 \\ 1 & 0 \end{pmatrix}\begin{pmatrix} z^{-1} & 0 
\\ 0 & w^{-1} \end{pmatrix}\right] w''{}^{-1} \notag \\
&= \frac1{2w''}\det\begin{pmatrix} 2z''+z^{-1} & 2w''+w^{-1}
\\ z''+z^{-1} & w'' \end{pmatrix} \notag \\
&= \frac1{2w''}\det\begin{pmatrix} z''+1 & w''+1\\ 1 & w'' 
\end{pmatrix} \notag\\ 
&=\frac12(z''-w''{}^{-1}) = \frac12{\sqrt{-3}}\,, \label{tor41}
\end{align}
where at intermediate steps we used $z''+z^{-1}-1=w''+w^{-1}-1=0$, and at the 
end we substituted the discrete faithful solution $z=z''=w=w''=e^{i\pi/3}$.

The invariant $\tau_{\mb{4_1}}$ belongs to the invariant trace field 
$E_{\mb{4_1}} = \mathbb{Q}(\sqrt{-3})$, and agrees with the torsion of the 
figure-eight knot complement \cite{DFJ}.

\subsection{Independence of a choice of quad type}
\lbl{sec.invgauge}

Now, let us fix a manifold $M$, a triangulation $\CT$ with $N$ tetrahedra, and an enhanced Neumann-Zagier datum $\widehat\beta_\CT=(z,\mb A,\mb B,f)$.

To prove independence of quad type, it is sufficient to check that $\tau_\CT$ is invariant under a cyclic permutation of the first triple of shape parameters $(z_1,z_1',z_1'')$, while holding fixed the choice of meridian loop and redundant edge.
Let us write $z=(z_1,\dots,z_N)$, $\eta=(\eta_1,\dots,\eta_N)^T$, $f=(f_1,...,f_N)^T$, and 
\be \mb A = \big(a_1,\,a_2,\,\cdots,a_N\big)\,,\quad \mb B 
= \big(b_1,\,b_2,\,\cdots,b_N\big)\,,\ee
in column notation. After the permutation, a new Neumann-Zagier datum is given by
$(\wt z, \wt{\mb A} ,\,\wt {\mb B},\wt f )$ where
\be \wt z = (z_1',z_2,...,z_N)\,,\quad \wt z'=(z_1'',z_2',...,z_N')\,,
\quad \wt z'' = (z_1,z_2'',...,z_N'')\,,\ee
\be \wt{\mb A} = \big(-b_1,\,a_2,\,\cdots,a_N\big)\,,\quad \wt{\mb B} 
= \big(a_1-b_1,\,b_2,\,\cdots,b_N\big)\,,\quad  \wt 
\eta = (n_1-b_1,n_2,...,n_N)^T\,.
\ee
The new shapes satisfy 
$\wt z^{\tilde{\mb A}}\wt z''^{\tilde{\mb B}}=(-1)^{\tilde \eta}$.
We also naturally obtain a new flattening $(\wt f,\wt f',\wt f'')$ by permuting
\be \wt f=(f_1',f_2,...,f_N)^T\,,\quad \wt f'=(f_1'',f_2',...,f_N')^T\,,
\quad \wt f''=(f_1,f_2'',...,f_N'')^T\,; \ee
this is an integer solution to 
$\wt{\mb A}\tilde f+\wt{\mb B}\tilde f''= \wt \eta$ and 
$\tilde f+\tilde f'+\tilde f''=1$.  

The torsion $\tau_{\CT}$ \eqref{torintro} 
consists of two parts, a determinant and a monomial 
correction. By making use of the relations $z_1+z_1'{}^{-1}-1=0$ and 
$z_1z_1'z_1''=-1$, we find the determinant with the permuted Neumann-Zagier datum to be
\begin{align} 
\det\big(\wt {\mb A}\Delta_{\tilde z''}+\wt {\mb B}\Delta_{\tilde z}^{-1}\big) &= 
\det \big( -b_1z_1 +(a_1-b_1)z_1'^{-1},\; a_2 z_2''+b_2 z_2^{-1} ,\;\cdots ,\; 
a_N z_N''+b_N z_N^{-1} \big) \notag\\
 &= \det\big( a_1 z_1'^{-1}-b_1 ,\; a_2 z_2''+b_2 z_2^{-1} ,\;\cdots ,\; 
a_N z_N''+b_N z_N^{-1} \big) \notag\\
 &= -z_1 \det\big( a_1 z_1''+b_1 z_1^{-1} ,\; a_2 z_2''+b_2 z_2^{-1} ,\;
\cdots ,\; a_N z_N''+b_N z_N^{-1} \big) \notag\\
 &= -z_1\det\big(\mb A\Delta_{z''}+\mb B\Delta_{z}^{-1}\big)\,.
\end{align}
By simply using $z_1z_1'z_1''=-1$ and $f_1+f_1'+f_1''=1$, we also see that 
the monomial correction transforms as
\begin{align} \wt z^{\tilde f''}\wt z''{}^{-\tilde f} &=
 z^{f''}z''{}^{-f} \frac{ z_1'{}^{f_1}z_1{}^{-f_1'} }{z_1{}^{f_1''}z_1''{}^{-f_1}} 
\notag\\
 &= z^{f''}z''{}^{-f} 
(-1)^{f_1}\frac{(z_1z_1'')^{-f_1}z_1{}^{f_1+f_1''-1}}{z_1{}^{f_1''}z_1''{}^{-f_1}} 
\notag \\
 &= z^{f''}z''{}^{-f}\,(-1)^{f_1}z_1^{-1}\,.
\end{align}
The extra factors $z_1^{\pm 1}$ in the two parts of the torsion precisely cancel 
each other, leading in the end to
\be \det\big(\wt {\mb A}\Delta_{\tilde z''}
+\wt {\mb B}\Delta_{\tilde z}^{-1}\big)\wt z^{\tilde f''}
\wt z''{}^{-\tilde f} = (-1)^{f_1+1}\det\big(\mb A\Delta_{z''}
+\mb B\Delta_{z}^{-1}\big)z^{f''}z''{}^{-f}\,.\ee
This is just as desired, showing that the torsion is invariant up to a sign.

\subsection{Independence of a choice of edge}
\lbl{sec.invglue}

We fix $M,\,\CT,\,\widehat\beta_\CT=(z,\mb A,\mb B,f)$.
In order to choose matrices $\mb A,\,\mb B$, we must ignore the redundant gluing 
equation corresponding to an edge of $\calT$. This was discussed in Section \ref{sec.NZ}.
Suppose, then, that we choose a different edge to ignore. For 
example, if we choose the $(N-1)^{\text{st}}$ rather than the $N^{\text{th}}$ (and keep the same quad type and meridian path), 
then we obtain new Neumann-Zagier matrices $\wt{\mb A},\,\wt{\mb B}$, 
which are related to the original ones as
\be \wt{\mb A} =  P_{(N-1,N)}\,\,\mb A\,,\qquad \wt{\mb B} 
=  P_{(N-1,N)}\,\,\mb B\,,\ee
where 
\be  P_{(N-1,N)} := \begin{pmatrix} 1 & 0 & \cdots & 0 & 0 & 0 \\
 0 & 1 & \cdots & 0 & 0 & 0 \\
 && \ddots &&&\\
 0 & 0 &\cdots & 1 & 0 & 0 \\
 -1 & -1 & \cdots & -1 & -1 & -1 \\
 0 & 0 & \cdots & 0 & 0 & 1
\end{pmatrix}\,.
\ee
Similarly, eliminating the $I^{\text{th}}$ rather than the $N^{\text{th}}$ 
edge constraint is implemented by multiplying with a matrix 
$P_{(I,N)}$ whose $I^{\rm th}$ row is filled with $-1$'s. Any such matrix satisfies  $\det P_{(I,N)}=- 1$.

In the formula for $\tau_{\CT}$, only the determinant part is affected by a change of edge. Then
\be \det\big(\wt{\mb A}\Delta_{z''}+\wt{\mb B}\Delta_{z}^{-1}\big)
 = \det\big(P_{(I,N)}({\mb A}\Delta_{z''}+{\mb B}\Delta_{z}^{-1})\big)
 = - \det\big({\mb A}\Delta_{z''}+{\mb B}\Delta_{z}^{-1}\big)\,,
\ee
leading to invariance of $\tau_{\CT}$, up to the usual sign.

\subsection{Independence of a choice of meridian path}
\lbl{sec.invmerid}

Recall that an ideal triangulation on $M$ induces a triangulation of its
boundary torus $\pt M$. Consider two simple closed meridian loops in $\pt M$ 
in general (normal) position with respect to the triangulation of $\pt M$.
Recall that these paths are drawn on the triangulated 2-dimensional torus 
$\pt M$ where faces of tetrahedra correspond 
to edges in the 2-dimensional triangulation, and edges of tetrahedra to 
vertices. In particular, for a one-cusped manifold $M$, every 
edge of the triangulation intersects a pair of vertices on the boundary $\pt M$.

\begin{figure}[htb]
\centering
\includegraphics[width=4.5in]{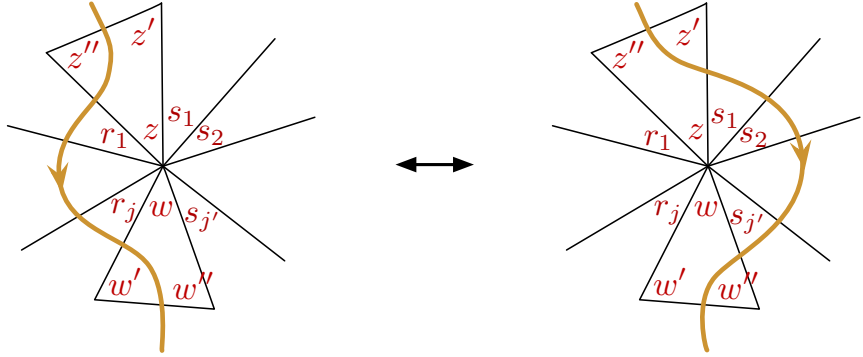}
\caption{The fundamental move for changing a meridian path. Here, we 
deform through an edge $E_I$ with gluing constraint 
$X_I = Z+W+R_1+...+R_j+S_1+...+S_{j'}=2\pi i$.}
\lbl{fig.meridmove}
\end{figure}

We can deform one of our meridian paths into the other by using repeated 
applications of the fundamental move shown in Figure \ref{fig.meridmove} 
--- locally pushing a section of the path across a vertex of $\pt M$. 
Thus, it suffices to assume that the two paths only differ by one such move. 
Suppose that we cross the $I^{\text{th}}$ edge (by Section 
\ref{sec.invglue} we may assume that $I\neq N$), which has a combinatorial 
gluing constraint
\be 
X_I := 
\sum_{i=1}^N \big(\mb{G}_{Ii}Z_i+\mb{G}_{Ii}'Z_i'+\mb{G}_{Ii}''Z_i''\big) 
= 2\pi i\,, 
\ee
and that the two tetrahedra where the paths enter and exit the vicinity of 
the edge have parameters $(Z,Z',Z'')$ and $(W,W',W'')$, as in the figure.
We do not exclude the possibility that $(Z,Z',Z'')$ and 
$(W,W',W'')$ both coincide with the same triple $(Z_i,Z_i',Z_i'')$, in 
some cyclic permutation.
Then the difference in the logarithmic meridian equations 
\eqref{merid} for the two paths will be
\be \delta_\mu = \pm \big(X_I - (Z+Z'+Z'')-(W+W'+W'')\big)\,. 
\lbl{meridchange} 
\ee
Note that two logarithmic meridian constraints that differ by 
\eqref{meridchange} are compatible and equivalent, since upon using the 
additional equations $X_I=2\pi i$ and $Z+Z'+Z''=W+W'+W''=i\pi$, we find 
that $\delta_\mu = 0$. A discretized version of this observation demonstrates 
that that the same flattening satisfies both discretized meridian constraints.%
\footnote{Note that this would not be the case if we allowed self-intersections of the 
meridian loops.}

If we compute matrices $\mb A,\,\mb B$ using one meridian path and 
$\wt{\mb A},\,\wt{\mb B}$ using the other --- keeping quad type, flattening, 
and edge the same --- the change 
\eqref{meridchange} implies that
\be 
\wt{\mb A} = P^{(\mu)}_I{}^{\pm 1}\mb A\,,
\qquad \wt{\mb B} = P^{(\mu)}_I{}^{\pm 1}\mb B\,,
\ee
where $P^{(\mu)}_I$ is the $\SL(N,\Z)$ matrix
\be P^{(\mu)}_I = I + E_{NI}\,, \ee
\ie\ the identity plus an extra entry `1' in the $N^{\text{th}}$ (meridian) 
row and $I^{\text{th}}$ column. Since $\det P^{(\mu)}_I=1$, this immediately shows 
that $\det\big(\wt{\mb A}\Delta_{z''}+\wt{\mb B}\Delta_{z}^{-1}\big)
= \det\big({\mb A}\Delta_{z''}+{\mb B}\Delta_{z}^{-1}\big)$, and so the a change 
in the meridian path cannot affect $\tau_{\CT}$.

\subsection{Independence of a choice of flattening}
\lbl{sec.invflat}

Now suppose that we choose two flattenings $(f,f',f'')$ 
and $(\wt f,\wt f',\wt f'')$, both satisfying
\be 
\mb A  f  + \mb B f'' = \eta\,,\qquad 
f+ f'+ f''=1\,. \ee
\be 
\mb A \wt f  + \mb B\wt f'' = \eta\,,\qquad 
\wt f+\wt f'+\wt f''=1\,. \ee
We may assume that we have a quad type with $\mb B$ non-degenerate. Indeed, by the result of Section \ref{sec.invgauge}, flattening invariance in one quad type implies flattening invariance in any quad type. Moreover, by Lemma \ref{lemma.B} of Appendix \ref{app.symp}, a quad type with non-degenerate $\mb B$ always exists. We also note that when $\mb B$ is invertible the matrix $\mb B^{-1}\mb A$ is symmetric (Lemma \ref{lemma.BA}).

The determinant in $\tau_{\CT}$ is insensitive to the change of 
flattening. The monomial, on the other hand, can be manipulated as follows. 
Let us choose logarithms $(Z,Z',Z'')$ of the shape parameters such that 
$\mb A Z+\mb B Z''=i\pi \eta$. 
Then, assuming that $\mb B$ is non-degenerate, we compute:
\begin{align} \frac{z^{\tilde f''}z''^{-\tilde f}}{z^{f''}z''{}^{-f}}
 &= \exp\big[ Z\cdot (f''-\wt f'')-Z''\cdot (f-\wt f) \big] \notag\\
 &= \exp\big[ -Z\cdot \mb B^{-1}\mb A(f-\wt f)
-(i\pi \mb B^{-1}\eta-\mb B^{-1}\mb A Z)\cdot (f-\wt f)\big] \notag\\
 &= \exp\big[-i\pi \mb B^{-1}\eta\cdot (f-\wt f)\big] \notag \\
 &= \exp\big[-i\pi f''\cdot(f-\wt f)
-i\pi \mb B^{-1}\mb A f\cdot(f-\wt f)\big] \notag\\
 &= \exp\big[-i\pi f''\cdot(f-\wt f)+i\pi  f\cdot(f''-\wt f'')\big] \notag \\
 &= \exp\big[i\pi (f''\cdot \wt f- f\cdot\wt f'')\big] = \pm 1\,. \notag
\end{align}
Therefore, the monomial can change at most by a sign, and $\tau_{\CT}$ is 
invariant as desired.

This completes the proof of Theorem \ref{thm.indep}.
\qed

\subsection{Invariance under 2--3 moves}
\lbl{sec.inv23}

We finally come to the proof of Theorem \ref{thm.1}, \ie the invariance of 
$\tau_{\CT}$ under 2--3 moves. We set up the problem as in Figure 
\ref{fig.23}. Namely, we suppose that 
$M$  
has two different regular triangulations $\CT$ and $\wt \CT$, with $N$ and $N+1$ tetrahedra, respectively, which are related by a local 2--3 move. Let us denote the 
respective (triples of) shape parameters as
\be 
Z := (X_1,X_2,Z_3,...,Z_N)\,,\qquad 
\tilde Z := (W_1,W_2,W_3,Z_3,...,Z_N)\,.
\ee
We fix a quad type, labeling the five tetrahedra involved in the 2--3 move 
as in the figure. We will also assume that when calculating Neumann-Zagier 
matrices $\mb A$ and $\mb B$, we choose to ignore an edge that is 
\emph{not} the central one of the 2--3 {\em bipyramid}.

\begin{figure}[htb]
\centering
\includegraphics[width=5.5in]{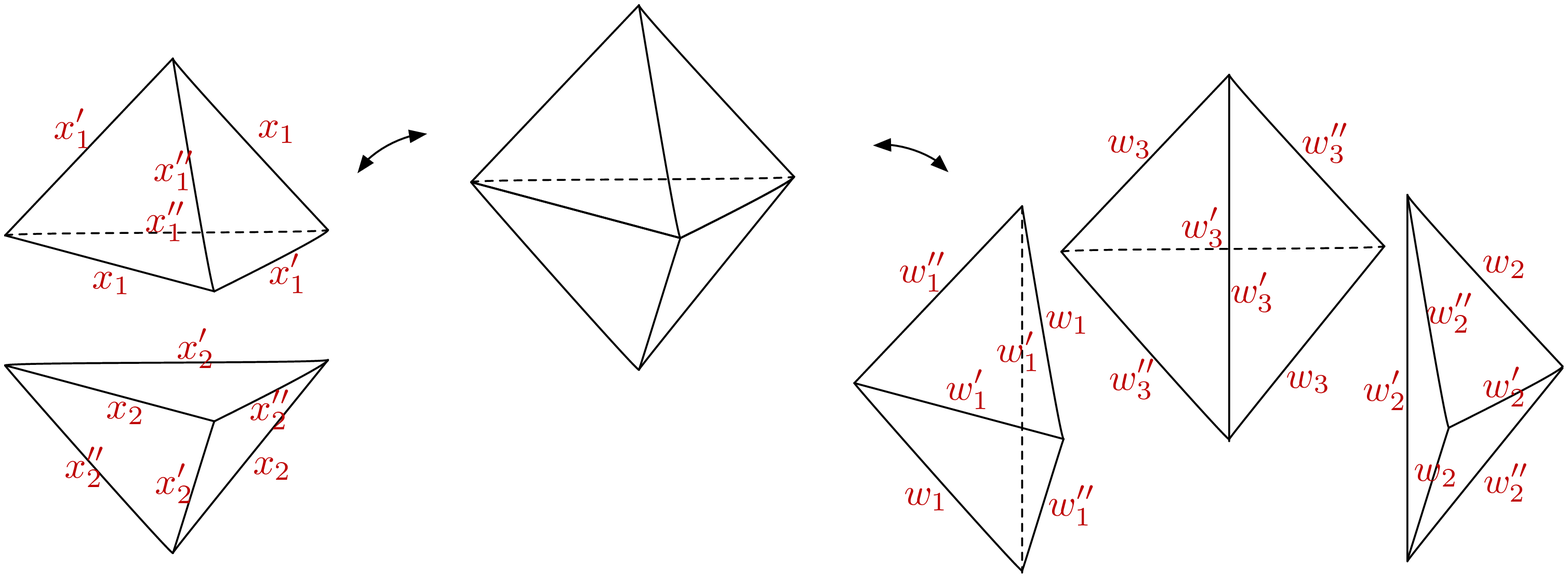}
\caption{The geometry of the 2--3 move: a bipyramid split into two 
tetrahedra for triangulation $\bigcup_{i=1}^N\Delta_i$, and three for 
triangulation $\bigcup_{i=1}^{N+1} \wt\Delta_i$.}
\lbl{fig.23}
\end{figure}

There are nine linear relations among the shapes of the tetrahedra involved 
in the move; three come from adding dihedral angles on the equatorial edges 
of the bipyramid:
\be 
W_1'=X_1+X_2\,,\qquad W_2'=X_1'+X_2''\,,\qquad 
W_3'= X_1''+X_2'\,, \lbl{eqs3}
\ee
and six from the longitudinal edges:
\be 
\lbl{eqs6}
\begin{array}{c@{\qquad}c@{\qquad}c}
X_1=W_2+W_3''\,,& X_1'=W_3+W_1''\,,&  X_1''=W_1+W_2''\,, \\[.2cm]
X_2 = W_2''+W_3\,,& X_2'=W_1''+W_2\,,& X_2''=W_3''+W_1\,.
\end{array}
\ee
Moreover, due to the central edge of the bipyramid, there is an extra 
gluing constraint in~$\wt \CT$:
\be 
W_1'+W_2'+W_3' = 2\pi i\,. 
\lbl{glue3} 
\ee

After exponentiating the relations (\ref{eqs3}--\ref{glue3}), and also 
using $z_iz_i'z_i''=-1$ and $z_i''+z_i^{-1}-1=0$ for every tetrahedron 
$\Delta_i$ and $\wt\Delta_i$, we find a birational map between the shape 
parameters in the two triangulations. Explicitly,
\begin{align} 
\lbl{map23}
\Big\{w_1'=x_1x_2,\,
w_2'=\frac{1-x_2^{-1}}{1-x_1},\,
w_3'=\frac{1-x_1^{-1}}{1-x_2}\Big\}\quad\text{or}\quad 
\Big\{ x_1=\frac{1-w_2'{}^{-1}}{1-w_3'},\,x_2=\frac{1-w_3'{}^{-1}}{1-w_2'}\Big\}\,.
\end{align}
Note that the birational map is well defined and one-to-one 
as long as no shape parameters $(x_1,x_2,w_1,w_2,w_3)$ equal 0, 1, or 
$\infty$. This condition is satisfied \emph{so long as} triangulations $\CT$ and $\wt \CT$ are both regular. (A necessary condition is that no univalent edges are created on one side or the other of the 2--3 move; this is also sufficient when considering the discrete faithful representation of $M$.)

We must also choose a flattening in the two triangulations. Let us suppose 
that for $\bigcup_{i=1}^{N+1}\wt \Delta_i$ we have a flattening with 
(triples of) integer parameters $\wt f=(d_1,d_2,d_3,f_3,...,f_N)$. This 
automatically determines a flattening $ f= (e_1,e_2,f_3,...,f_N)$ for the 
$\bigcup_{i=1}^{N} \Delta_i$ triangulation, by simply setting
\be 
\lbl{flat6}
\begin{array}{c@{\qquad}c@{\qquad}c}
e_1=d_2+d_3''\,,& e_1'=d_3+d_1''\,,&  e_1''=d_1+d_2''\,, \\[.2cm]
e_2 = d_2''+d_3\,,& e_2'=d_1''+d_2\,,& e_2''=d_3''+d_1\,.
\end{array}
\ee
This is a discretized version of the six longitudinal relations 
\eqref{eqs6}. One can check that expected relations such as 
$e_1+e_1'+e_1''=1$ are satisfied by virtue of the discretized 
edge constraint $d_1'+d_2'+d_3'=2$ (\cf\ \eqref{glue3}).

We have all the data needed to calculate $\tau_{\CT}$. Let us start with 
determinants. In the triangulation $\bigcup_{i=1}^{N} \Delta_i$, we write 
the matrices $\mb A$ and $\mb B$ schematically in columns as
\be 
\mb A = (a_1,\;a_2,\;a_i)\,,\quad \mb B = (b_1,\;b_2,\;b_i)\,, 
\ee
with $a_i$ meaning $(a_3,a_4,...,a_N)$ and similarly for $b_i$. This leads 
to a determinant
\be \det\big({\mb A}\Delta_{z''}+{\mb B}\Delta_{z}^{-1}\big) 
= \det\Big(a_1x_1''+\tfrac{b_1}{x_1},\; a_2x_2''+\tfrac{b_2}{x_2},\; 
a_iz_i''+\tfrac{b_i}{z_i} \big)\,.
\ee
Alternatively, in the triangulation $\bigcup_{i=1}^{N+1}\wt\Delta_i$, the 
matrices $\wt{\mb A}$ and $\wt{\mb B}$ have one extra row and one extra 
column. The extra gluing condition \eqref{glue3} causes the extra row in 
both $\wt{\mb A}$ and $\wt{\mb B}$ to contain three $-1$'s. Altogether, 
the matrices take the form
\be 
\wt {\mb A} = \begin{pmatrix} -1 & -1 & -1  & 0 \\ 
b_1+b_2 & a_1 & a_2 & a_i \end{pmatrix}\,,\qquad \wt{\mb B} 
=\begin{pmatrix} -1 & -1 & -1 & 0 \\ 0 & a_2+b_1 & a_1+b_2 & b_i 
\end{pmatrix}\,,
\ee
so that
\begin{align} 
\wt{\mb A}\Delta_{\tilde z''}+\wt{\mb B}\Delta_{\tilde z}^{-1} 
&= \begin{pmatrix} -w_1''-\frac{1}{w_1} & -w_2''-\frac{1}{w_2  } 
& -w_3''-\frac{1}{w_3} & 0 \\ (b_1+b_2)w_1'' & a_1w_2''+\frac{a_2+b_1}{w_2} 
& a_2 w_3''+\frac{a_1+b_2}{w_3} & a_i z_i''+\frac{b_i}{z_i} \end{pmatrix} 
\notag\\
&=  \begin{pmatrix} -1 & -1 & -1 & 0 \\ (b_1+b_2)w_1'' 
& a_1w_2''+\frac{a_2+b_1}{w_2} & a_2 w_3''+\frac{a_1+b_2}{w_3} 
& a_i z_i''+\frac{b_i}{z_i} \end{pmatrix}\,. \notag
\end{align}
It is then straightforward to check, using the map \eqref{map23}, that
\be
\big(\wt{\mb A}\Delta_{\tilde z''}+\wt{\mb B}\Delta_{\tilde z}^{-1}\big) 
\begin{pmatrix} 1 & -1 & -1 & 0 \\
   0 & 1 & 0 & 0 \\
   0 & 0 & 1 & 0 \\
   0 & 0 & 0 & \mb{1} \end{pmatrix}
 \begin{pmatrix} 1 & 0 & 0 & 0 \\
   0 & 1 & \frac{w_3' x_2''}{x_1''} & 0 \\
   0 & \frac{w_2'x_1''}{x_2''} & 1 & 0  \\
   0 & 0 & 0 & \mb{1} \end{pmatrix}=
 \begin{pmatrix} -1 & 0 \\
   * & {\mb A}\Delta_{z''}+{\mb B}\Delta_{z}^{-1} \end{pmatrix}
 \,.
\ee
The determinant of the last matrix on the left hand side is 
$1-w_2'w_3'=1-w_1'{}^{-1}=w_1$. Therefore,
\be \lbl{det23} \det\big(\wt{\mb A}\Delta_{\tilde z''}
+\wt{\mb B}\Delta_{\tilde z}^{-1}\big) 
= -w_1^{-1}\det\big({\mb A}\Delta_{z''}+{\mb B}\Delta_{z}^{-1}\big)\,.
\ee

We should also consider the monomial correction. However, with flattenings 
related as in \eqref{flat6}, and with shapes related by the exponentiated 
version of \eqref{eqs6}, it is easy to check that
\be 
\wt z^{\tilde f''}\wt z''{}^{-\tilde f} 
= w_1 (-1)^{d_2''-e_1''}z^{f''}z''{}^{-f}\,. 
\lbl{mon23} 
\ee 
We have thus arrived at the desired result; by combining \eqref{det23} 
and \eqref{mon23}, we find
\be 
\det\big(\wt{\mb A}\Delta_{\tilde z''}
+\wt{\mb B}\Delta_{\tilde z}^{-1}\big)\,
\wt z^{\tilde f''}\wt z''{}^{-\tilde f} = \pm \det \big({\mb A}\Delta_{z''}
+{\mb B}\Delta_{z}^{-1}\big)z^{f''}z''{}^{-f}\,,
\ee
so that $\tau_{\CT}$ is invariant under the 2--3 move. This completes the
proof of Theorem \ref{thm.1}.
\qed


\section{Torsion on the character variety}
\lbl{sec.extensions}

Having given a putative formula for the non-abelian torsion of a cusped 
hyperbolic manifold $M$ at the discrete faithful representation $\rho_0$, 
it is natural to ask whether the formula generalizes to other settings. In 
this section, we extend the torsion formula to general representations 
$\rho:\pi_1(M)\longto \text{(P)SL}(2,\BC)$ for manifolds $M$ with torus 
boundary, essentially by letting the shapes 
$z$ be functions of $\rho$. We also find that some special results hold 
when $M$ is hyperbolic and the representations lie on the geometric 
component $X_M^{\rm geom}$ of the $SL(2,\C)$ character variety.

We will begin with a short review of what it means for a combinatorial 
ideal triangulation to be \emph{regular} with respect to a general 
representation $\rho$. We will also finally prove Proposition \ref{prop.EP} 
from the Introduction. Recall that Proposition \ref{prop.EP} identified a 
canonical connected subset $\calX^{\EP}_M$ of the set of regular triangulations 
$\calX_{\rho_0}$ of a hyperbolic 3-manifold $M$. This result allowed us to 
construct the topological invariant $\tau_M$.

We then proceed to define an enhanced Neumann-Zagier datum 
$\widehat\beta_\CT=(z,\mb A,\mb B,f)$ suitable for a general representation 
$\rho$, and propose a generalization of the torsion formula:
\be 
\lbl{utor} 
\tau_\CT(\rho) := \pm\frac12\det\big(\mb A\Delta_{z''}
+\mb B\Delta_z^{-1}\big)z^{f''}z''{}^{-f}  \,.
\ee
This formula looks identical to \eqref{torintro}. However, the shape parameters here are promoted to functions $z\to z(\rho)$ of the representation $\rho$, which satisfy a well known deformed version of the gluing equations. Moreover, the flattening $f$ is slightly more restricted than it was previously. 
We will prove in Section \ref{sec.utorsion} that

\begin{theorem}\lbl{thm.utor}
The formula for $\tau_\CT(\rho)$ is independent of the choice of enhanced Neumann-Zagier datum, and is invariant under 2--3 moves connecting $\rho$-regular triangulations.
\end{theorem}

\noindent When $M$ is hyperbolic, it turns out that $\rho_0$-regular triangulations are $\rho$-regular for all but finitely many representations $\rho\in X_M^{\rm geom}$. Then we can create a topological invariant $\tau_M$ that is a function on $X_M^{\rm geom}$ just as in Proposition \ref{prop.EP}, by evaluating $\tau_\CT(\rho)$ on any triangulation in the canonical subset $\calX^{\EP}_M\subseteq \calX_{\rho_0}$. 

In general, there is a rational map from the character variety $X_M$ to the zero-locus $Y_M$ of the A-polynomial $A_M(\ell,m)=0$ \cite{CCGLS}, for any $M$ with torus boundary. Therefore, the shapes $z$ and the torsion $\tau_\CT$ are algebraic functions on components of the A-polynomial curve $Y_M$. When $M$ is hyperbolic and $\rho\in X_M^{\rm geom}$, somewhat more is true: the shapes are \emph{rational} function on the geometric component $Y_M^{\rm geom}$ (Proposition \ref{prop.shapes}). Then
\be \qquad  \tau_M \;\in\; C(Y_M^{\rm geom}) = \BQ(m)[\ell]/\big(A_M^{\rm geom}(\ell,m)\big)\,,
\ee
where $A^{\rm geom}_M(\ell,m)$ is the geometric factor of the A-polynomial. We will give a simple example of the function $\tau_M$ for the figure-eight knot in Section \ref{sec.utor41}.

\subsection{A review of $\rho$-regular ideal triangulations}
\lbl{sub.revrho}

In this section we discuss the $\rho$-regular ideal triangulations
that are needed to generalize our torsion invariant. 
Let $M$ denote a 3-manifold with nonempty boundary and let
$\rho: \pi_1(M)\longto\PSL(2,\BC)$ be a $\PSL(2,\BC)$ representation of its
fundamental group. Let $\calX$ denote the set of combinatorial ideal
triangulations $\calT$ of $M$. 
Matveev and Piergallini independently
showed that every two elements of $\calX$ with at least two ideal
tetrahedra are connected by a sequence of 2-3 moves (and their inverses)
\cite{Ma1,Pi}. For a detailed exposition, see \cite{Ma2,BP}.

Given an ideal triangulation $\calT$, let $V_{\calT}$ denote the affine variety
of non-degenerate solutions (i.e., solutions in $\BC\setminus\{0,1\}$)
of the gluing equations of $\calT$ corresponding
to its edges. There is a developing map
\be
\lbl{eq.dev}
V_{\calT} \longto X_M\,,
\ee
where $X_M:=\Hom(\pi_1(M),\PSL(2,\BC))/\PSL(2,\BC)$ denotes the affine variety
of all $\PSL(2,\BC)$ representations of $\pi_1(M)$.

\begin{definition}
\lbl{def.rhoregular}
Fix a $\PSL(2,\BC)$-representation of $M$. We say that $\calT \in \calX$
is $\rho$-{\em regular} if $\rho$ is in the image of the developing map
\eqref{eq.dev}. 
\end{definition}
Let $\calX_\rho \subset \calX$ denote the set of all $\rho$-regular
ideal triangulations of $M$. When $M$ is hyperbolic, let
$\rho_0$ denote its discrete faithful 
representation $\rho_0$ and let $X^{\geom}_M \subset X_M$ denote the geometric 
component of its character variety \cite{Th,NZ}. We then have the following result.

\begin{lemma}
\lbl{lem.wellknown}
\rm{(a)} $\calT \in \calX_{\rho_0}$ if and only if 
$\calT$ has no homotopically peripheral (\ie, univalent) edges.
\newline
\rm{(b)} If $\calT \in \calX_{\rho_0}$, then 
$\calT \in \calX_{\rho}$ for all but finitely many $\rho \in X^{\geom}_M$.
\end{lemma}

\begin{proof}
Part (a) has been observed several times; see \cite{Champanerkar},
\cite[Sec.10.3]{BDR-V}, \cite[Thm.2.3]{Till} and also \cite[Rem.3.4]{DnG}.
For part (b), fix $\calT \in \calX_{\rho_0}$. Observe that $\calT$ is
$\rho$-regular if the image of 
every edge%
\footnote{Note that every edge can be completed to a closed loop by adding a 
path on the boundary $T^2$. The choice of completion does not matter for 
studying commutation with the peripheral subgroup.} %
of $\calT$ under $\rho$ does not commute with the image under
$\rho$ of the peripheral subgroup of $M$. This is an algebraic 
condition on $\rho$, and moreover, when $\rho \in X^{\geom}_M$ is analytically
nearby $\rho_0$, the condition is satisfied. It follows that the set
of points of $ X^{\geom}_M$ that satisfy the above condition is Zariski
open. On the other hand, $X^{\geom}_M$ is an affine curve \cite{Th,NZ}.
It follows that $\calT$ is $\rho$-regular for all but finitely many
$\rho \in  X^{\geom}_M$. 
\end{proof}

\subsection{The Epstein-Penner cell decomposition and its triangulations}
\lbl{sub.EP}

Now we consider the canonical ideal cell decomposition of
a hyperbolic manifold $M$ with cusps \cite{EP}, and finally prove Proposition \ref{prop.EP}. It is easy to
see that every convex ideal polyhedron can be triangulated into ideal
tetrahedra with non-degenerate shapes, see for instance \cite{HRS}.
One wishes to know that every two such triangulations are related by
a sequence of 2--3 moves. This is a combinatorial problem of convex geometry
which we summarize below. For a detailed discussion,  
the reader may consult the book \cite{LRS} and references therein.

Fix a convex polytope $P$ in $\BR^d$. One can consider the set of
triangulations of $P$. When $d=2$, $P$ is a polygon and it is known that 
every two triangulations are related by a sequence of flips. For general
$d$, flips are replaced by {\em geometric bistellar moves}. When $d \geq 5$,
it is known that the graph of triangulations (with edges given
by geometric bistellar flips) is not connected, and has isolated vertices.
For $d=3$, it is not known whether the graph is connected.

The situation is much better when one considers {\em regular triangulations}
of $P$. In that case, the corresponding graph of regular triangulations
is connected, an in fact it is the edge set of the 
{\em secondary polytope} of $P$. When $d=3$ and $P$ is convex and in general 
position, then the only geometric bistellar move is the 2--3 move where
the added edge that appears in the move is an edge that connects two vertices
of $P$. When $d=3$ and $P$ is not in general position, the same conclusion
holds as long as one allows for tetrahedra that are flat, \ie, lie
on a 2-dimensional plane.

Returning to the Epstein-Penner ideal cell decomposition, let
$\calX^{\EP}_M$ denote the set of regular (in the sense of polytopes and
in the sense of $\rho_0$) ideal triangulations of the ideal cell decomposition.
The above discussion together with the fact that no edge of the ideal cell
decomposition is univalent, implies that $\calX^{\EP}$ is a connected
subset of $\calX_{\rho_0}$. This concludes the proof of Proposition
\ref{prop.EP}.

A detailed discussion on the canonical set $\calX^{\EP}_M$ of ideal 
triangulations of a cusped hyperbolic 3-manifold $M$ is given in 
\cite[Sec.6]{GHRS}.


\subsection{Neumann-Zagier datum and the geometric component}
\lbl{sub.GE}

Let $M$ be a manifold with torus boundary and $\CT$ a (combinatorial) 
ideal triangulation. The Neumann-Zagier datum 
$\beta_\CT=(z,\mb A,\mb B)$ may be generalized for representations $\rho\in X_M$ 
besides the discrete faithful.

To begin, choose a representation $\rho:\pi_1(M)\longto \PSL(2,\BC)$, and, if desired, a lift to ${\rm SL}(2,\BC)$. Let $(\mu,\lambda)$ be meridian and longitude cycles%
\footnote{Recall again that these cycles are only canonically defined for knot complements. In general there is some freedom in choosing them, but the torsion depends in a predictable way on the choice, \cf\ \cite{Ya}.} %
on $\pt M$, and let $(m^{\pm 1},\ell^{\pm 1})$ be the eigenvalues of $\rho(\mu)$ and $\rho(\lambda)$, respectively. For example, for the lift of the discrete faithful representation to ${\rm SL}(2,\C)$, we have $(m,\ell)=(1,-1)$ \cite{Calegari}. These eigenvalues define a map
\be \lbl{Amap} X_M \longto (\BC^*)^2/\BZ_2\,,\ee
whose image is a curve $Y_M$, the zero-locus of the A-polynomial $A_M(\ell,m)=0$ \cite{CCGLS}. 
We will denote the representation $\rho$ as $\rho_m$ to emphasize its meridian eigenvalue.

Now, given a triangulation $\CT$, and with $\mb A$, $\mb B$, and $\eta$ defined as in Section \ref{sec.NZ}, the gluing equations \eqref{glueeqs} can easily be deformed to account for $m\neq 1$. Namely, we find \cite{NZ}
\be 
\lbl{uglue} 
\prod_{j=1}^Nz_j{}^{\mb A_{ij}}z_j''{}^{\mb B_{ij}} = (-1)^{\eta_i} m^{2\delta_{iN}}\,. 
\ee
The developing map \eqref{eq.dev} maps every solution of these equations to a representation $\pi_1(M)\longto \PSL(2,\BC)$ with meridian eigenvalue $\pm m$. The triangulation $\CT$ is $\rho_m$-regular if and only if $\rho_m$ is in the image of this map. We can similarly express the longitude eigenvalue as a product of shape parameters
\be 
\lbl{uglueL} 
\prod_{j=1}^N z_j{}^{2 C_j}z_j''{}^{2D_j} = (-1)^{2\eta_\lambda}\ell^2\,,\ee
for some $2 C_j,2 D_j,2\eta_\lambda \in \BZ$.
Then, if $\CT$ is a $\rho_m$-regular triangulation, the irreducible component of $Y_M$ containing $\rho_m$ is explicitly obtained by eliminating all shapes $z_j$ from \eqref{uglue}--\eqref{uglueL}. 

In general, the shapes $z_j$ are algebraic functions on components of the variety $Y_M$. However, if $M$ is hyperbolic and $\CT$ is regular for all but finitely many representations on the geometric component $Y_M^{\rm geom}$, then the shapes $z_j$ become rational functions, $z_j\in C(Y_M^{\rm geom})$. We provide a proof of this fact in Appendix \ref{app.shapes}. The field of functions $C(Y_M^{\rm geom})$ may be identified with $\BQ(m)[\ell]/\big(A^{\rm geom}(\ell,m)\big)$, and the functions $z_j(\ell,m)$ can easily be obtained from equations \eqref{uglue}--\eqref{uglueL}.

\subsection{Flattening compatible with a longitude}
\lbl{sec.uflat}

In this section we define a restricted combinatorial flattening that 
is compatible with a longitude.

Recall what is a combinatorial flattening of an ideal triangulation $\CT$ from
Definition \ref{def.flattening}.
Given a simple peripheral curve $\lambda$ on the boundary of $M$ that represents a longitude --- in particular, having intersection number one with the chosen meridian $\mu$ --- we can construct the sum of combinatorial edge parameters along $\lambda$, just as in Section \ref{sec.glue}. It takes the form
\be \lambda\;:\quad \sum_{i=1}^N \big( \mb G_{N+2,i}Z_i+\mb G_{N+2,i}'Z_i'
+\mb G_{N+2,i}''Z_i''\big)\,, \ee
for integer vectors $\mb G_{N+1}$, $\mb G'_{N+2}$, $\mb G''_{N+2}$. Just as we obtained $\mb A$, $\mb B$, and $\eta$ from the edge and meridian equations (with or without deformation), we may also now define
\be \label{defCD} C_i = \frac12(\mb G_{N+2,i}-\mb G_{N+2,i}')\,,\quad
    D_i = \frac12(\mb G_{N+2,i}''-\mb G_{N+2,i}')\,,\quad \eta_\lambda = -\frac12\sum_{i=1}^N \mb G_{N+2,i}'\,.
\ee
\begin{definition}
\lbl{def.flattening.long}
A combinatorial flattening $(f,f',f'')$ is compatible with a longitude if
in addition to Equations \eqref{flattrip}--\eqref{flatG}, it also satisfies
\be
\lbl{flatGlon}
\mb G_{N+2,i}f+\mb G_{N+2,i}'f'+\mb G_{N+2,i}''f''=0 \,.
\ee
Equivalently, a combinatorial flattening compatible with the longitude
is a vector $(f,f'') \in \BZ^{2N}$ that satisfies
\be \label{uflat} \mb A f+\mb B f'' = \eta\,,\qquad C\cdot f+D\cdot f'' = \eta_\lambda\,,\ee
\end{definition}
A combinatorial flattening compatible with a longitude 
always exists \cite[Lem.6.1]{Neumann-combi}.

In the context of functions on the character variety, it is natural to deform the meridian gluing equation, and simultaneously to introduce a longitude gluing equation, in the form
\begin{subequations}
\be  \lbl{umerid}
\mu\;:\quad 
\sum_{i=1}^N \big( \mb G_{N+1,i}Z_i+\mb G_{N+1,i}'Z_i'
+\mb G_{N+1,i}''Z_i''\big)=2u
\ee
\be 
\lbl{long}
\lambda\;:\quad \sum_{i=1}^N \big( \mb G_{N+2,i}Z_i+\mb G_{N+2,i}'Z_i'
+\mb G_{N+2,i}''Z_i''\big)=2v\,, 
\ee
\end{subequations}
for some complex parameters $u$ and $v$. Upon exponentiation, these equations reduce to the expected \eqref{uglue}--\eqref{uglueL} if
\be m^2 = e^{2u}\,,\qquad \ell^2 = e^{2v}\,. \ee
(It is easy to show, following \cite{NZ}, that $C$ and $D$ as defined by \eqref{uflat} are the correct exponents for the exponentiated longitude equation \eqref{uglueL}.) If $M$ is a knot complement and we want to lift from $PSL(2,\C)$ to $SL(2,\C)$ representations, we should take $m=e^u$ and $\ell=-e^v$ and divide \eqref{long} by two before exponentiating. This provides the correct way to take a square root of the exponentiated gluing equation, \cf\ \cite{Calegari}.

The remarkable symplectic property of $\mb A$ and $\mb B$ may be extended to $C$ and $D$, even though in general $C$, $D$ are vectors of 
half-integers rather than integers.
Namely, there exists a completion of $(\mb A\;\mb B)$ to a full symplectic matrix $\left(\begin{smallmatrix}\mb A & \mb B \\ \mb C & \mb D\end{smallmatrix}\right)$ such that the \emph{bottom rows} of $\mb C$ and $\mb D$ are the vectors $C$ and $D$ \cite{NZ}. In particular, this means that
\be \mb A_N\cdot D-\mb B_N\cdot C = 1\,, \ee
where $\mb A_N,\mb B_N$ are the bottom (meridian) rows of $\mb A,\mb B$.

\subsection{Invariance of the generalized torsion}
\label{sec.utorsion}

We finally have all the required ingredients for the generalized torsion 
formula. Let $M$ be a three-manifold with torus boundary, and 
$\rho_m:\pi_1(M)\to \text{(P)SL}(2,\C)$ a representation with meridian 
eigenvalue $m$. Let $\CT$ be a $\rho_m$-regular triangulation of $M$, 
which exists by Lemma \ref{lem.wellknown} at least for a dense set of 
representations on the geometric component of the character variety. 
Choose an enhanced Neumann-Zagier datum $(z,\mb A,\mb B,f)$, with
$z=z(\rho_m)$ satisfying the deformed gluing equations \eqref{uglue} and 
$f$ satisfying \eqref{uflat}. Then, as in \eqref{utor}, we define
\be 
\notag
\tau_\CT(\rho_m) := \pm\frac12\det\big(\mb A\Delta_{z''}
+\mb B\Delta_z^{-1}\big)z^{f''}z''{}^{-f}  \,.
\ee
We can now prove Theorem \ref{thm.utor}.

Repeating verbatim the arguments of Section \ref{sec.torsion}, it is easy
to see that $\tau_{\CT}$ is independent of a choice of quad type, a choice of an
edge of $\CT$ and a choice of a meridian loop. The crucial 
observation is that the equations $\mb AZ+\mb B Z''=i\pi \eta$ (including the meridian equation) are never 
used in the respective proofs. Therefore, deforming the meridian equation by $u\neq 0$ does not affect anything. For the same reason, 
it is not hard to see that the formula is invariant under $\rho_m$-regular 
2--3 moves, by repeating the argument of Section \ref{sec.inv23}.

The only nontrivial verification required is that $\tau_{\CT}$ is independent of the choice of flattening. This does use the gluing equations in a 
crucial way.  We check it now for $m\neq 1$.

Choose logarithms $(Z,Z',Z'')$ of the shape parameters and a logarithm $u$ of $m$ such that $Z+Z'+Z''=i\pi$ and
\be 
\lbl{uveceqs} 
\mb AZ+\mb BZ''=2\bm u+i\pi \eta\,, 
\ee
where $\bm{u}$ denote the $N$-dimensional vector $(0,0,....,0,u)^T$.
By independence of quad type and Lemma \ref{lemma.B}, we may assume we are using a quad type with non-degenerate $\mb B$.
Now, suppose that $(f,f',f'')$ and $(\wt f,\wt f',\wt f'')$ are two 
different generalized flattenings. Then:
\begin{align*} 
(Z\cdot f''-Z''\cdot f)-&(Z\cdot \wt f''-Z''\cdot \wt f) \\&=
 Z\cdot (f''-\wt f'') + \mb B^{-1}(\mb A Z-i\pi \eta -2\bm u)\cdot (f-\wt f) \\
 &= Z\cdot (f''-\wt f'') 
+  Z\cdot \mb B^{-1}\mb A(f-\wt f)-i\pi\mb B^{-1}\eta \cdot (f-\wt f) 
-2\mb B^{-1}\bm u\cdot(f-\tilde f) \\
 &= -i\pi\mb B^{-1}\eta\cdot (f-\wt f) -2\mb B^{-1}\bm u\cdot(f-\tilde f)\\
 &= i\pi (f''\cdot \wt f-f\cdot \wt f'') -2\mb B^{-1}\bm u\cdot(f-\tilde f)\,,
\end{align*}
by manipulations similar to those of Section \ref{sec.invflat}.
The new term 
$2\mb B^{-1}\bm u\cdot(f-\tilde f)$ is now dealt with with by completing the 
Neumann-Zagier matrices $(\mb A\;\mb B)$ to a full symplectic matrix 
$\left(\begin{smallmatrix} \mb A & \mb B \\ \mb C & 
\mb D\end{smallmatrix}\right) \in \Sp(2N,\mathbb Q)$, whose bottom row agrees with $(C,D)$. The symplectic 
condition implies that $\mb A \mb D^T-\mb B \mb C^T=I$, or 
$\mb B^{-1}=\mb B^{-1}\mb A\mb D^T-\mb C^T$. Then
\begin{align*} 
\mb B^{-1}\bm u\cdot(f-\tilde f) &=
 \mb B^{-1}\mb A\mb D^T \bm u\cdot (f-\tilde f) 
- \mb C^T \bm u\cdot (f-\tilde f) \\
 &= \bm u\cdot \mb D\mb B^{-1}\mb A (f-\tilde f) 
- \bm u\cdot \mb C(f-\tilde f) \\
 &= -\bm u\cdot\big( \mb C(f-\tilde f)+\mb D(f''-\tilde f'')\big) \\
 &= -u\big(C\cdot(f-\tilde f)+D\cdot (f''-\tilde f'')\big)\,.
\end{align*}
In this last equation, only the bottom row of $\mb C$ and $\mb D$ appears, 
due to the contraction with $\bm u = (0,0,...,0,u)$. But this bottom row is 
precisely what enters the generalized flattening equations \eqref{uflat}; since both 
flattenings satisfy these equations, we must have 
$\mb B^{-1}\bm u\cdot (f-\tilde f)=0$. Therefore, upon exponentiating, we find
\be z^{f''}z''{}^{-f} = (-1)^{f''\cdot \tilde f-f\cdot \tilde f''} 
z^{\tilde f''}z''{}^{-\tilde f} = \pm z^{\tilde f''}z''{}^{-\tilde f}\,,\ee
which demonstrates that $\tau_{\CT}$ is independent of the choice of flattening. Theorem \ref{thm.utor} follows.
\qed

\subsection{Example: $\mb 4_1$ continued}
\label{sec.utor41}

We briefly demonstrate the generalized torsion formula, using representations on the geometric component of the character variety $X_{\mb{4_1}}$ for the figure-eight knot complement.

\begin{figure}[htb]
\centering
\includegraphics[width=4.3in]{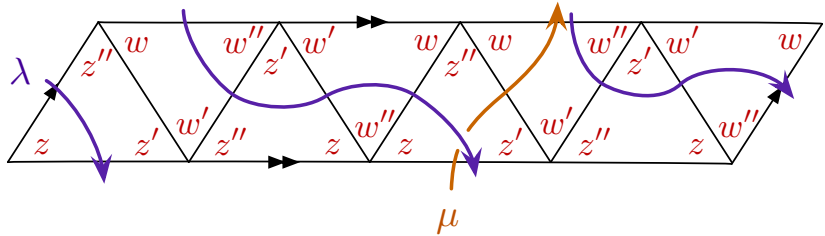}
\caption{Longitude path for the figure-eight knot complement.}
\label{fig.cusp41L}
\end{figure}

We may consider the same triangulation as in Section \ref{sec.41triang}. The edge and meridian equations \eqref{41glue} are deformed to
\be \label{41m} z^2w^2z''w''=1\,, \qquad zwz''=-m^2\,, \ee
with $z''=1-z^{-1}$, $w''=1-w^{-1}$ as usual. In addition, there is a longitude equation that may be read off from the longitude path in Figure \ref{fig.cusp41L}. In logarithmic (combinatorial) form, we have $-2Z+2Z'=2v$, or
\be \label{41L} -4Z-2Z''=2v-2\pi i\,,\ee
from which we identify
\be C = (-2,0)\,,\qquad D = (-1,0)\,,\qquad \eta_\lambda = -1.\ee
Dividing \eqref{41L} by two and exponentiating, we find
\be \label{41ell} z^{-2}z''{}^{-1}=\ell\,,\ee
with $\ell=-e^v$. This is the appropriate square root of \eqref{uglueL} for lifting the geometric representations to $\SL(2,\C)$. We can easily check it: by eliminating shape parameters from \eqref{41m} and \eqref{41ell}, we recover the geometric $\SL(2,\C)$ A-polynomial for the figure-eight knot,
\be A_{\mb{4_1}}^{\rm geom}(\ell,m) = m^4-(1-m^2-2m^4-m^6+m^8)\ell+m^4\ell^2\,. \ee

We may also use equations \eqref{41m}--\eqref{41ell} to express the shape parameters as functions of $\ell$ and $m$. We find
\be \label{zw41m} z = -\frac{m^2-m^{-2}}{1+m^2\ell} \,,\qquad w= \frac{m^2+\ell}{m^2-m^{-2}}\,.\ee
These are functions on the curve $Y_{\mb{4_1}}^{\rm geom}=\{A_{\mb{4_1}}^{\rm geom}(\ell,m)=0\}$.

The flattening \eqref{flat41} does \emph{not} satisfy the new longitude constraint $C\cdot f+D\cdot f''=\eta_\lambda$, so we must find one that does. The choice
\be (f_z,f_z',f_z'';f_w,f_w',f_w'') =(0,0,1;0,0,1) \ee
will work. Repeating the calculation of Section \ref{sub.torexample} with the same $\mb A$ and $\mb B$ but the new generalized flattening, we now obtain
\begin{align} \tau_{\mb{4_1}}(\rho_m) &= \pm \frac12 \det\begin{pmatrix} z''+1 & w''+1\\ 1 & w'' 
\end{pmatrix} z w \notag \\
&= \pm\frac12(z''w''-1)zw \notag \\
&= \pm\frac{1-m^2-2m^4-m^6+m^8-2m^4\ell}{2m^4(m^2-m^{-2})}\,. \label{41utor}
\end{align}
This is in full agreement with the torsion found by \cite{GuM, DGLZ}. Note that for fixed $m$ there are \emph{two} choices of representation $\rho_m$ on the geometric component of the character variety; they correspond to the two solutions of $A_{\mb{4_1}}^{\rm geom}(\ell,m)=0$ in $\ell$.

\begin{remark}
It is interesting to observe that the numerator of \eqref{41utor} is exactly $\pt A_{\mb{4_1}}^{\rm geom}/\pt \ell$. That the numerator of the geometric torsion typically carries a factor of $\pt A_{\mb{4_1}}^{\rm geom}/\pt \ell$ might be gleaned from the structure of ``$\hat A$-polynomials'' in \cite{DGLZ, GukovS}, and will also be explored elsewhere.
\end{remark}


\section{The state integral and higher loops}
\lbl{sec.SIM}

Our explicit formulas for the torsion $\tau_\CT$, as well as higher invariants $S_{\CT,n}$, have been obtained from a state integral model for analytically continued $\SL(2,\C)$ Chern-Simons theory. In this section, we will review 
the state integral, and analyze its asymptotics in order to re-derive the full asymptotic expansion
\be \label{asympSIM}
\CZ_\CT(\hbar) = \hbar^{-\frac32}\exp\Big[\frac1\hbar S_{\CT,0}+S_{\CT,1}+\hbar S_{\CT,2}+\hbar^2 S_{\CT,3}+\ldots\Big]\,,
\ee
and to unify the formulas of previous sections. We should point out that
this section is {\em not} analytically rigorous, but serves as a motivation
for our definition of the all-loop invariants, and provides a glimpse
into the calculus of (complex, finite dimensional) state-integrals.

The basic idea of a state integral is to cut a 
manifold $M$ into canonical pieces (ideal tetrahedra); to assign a simple 
partition function to each piece (a quantum dilogarithm); and then to 
multiply these simple partition functions together and integrate out over 
boundary conditions in order to obtain the partition function of the glued 
manifold $M$. A state integral provides a \emph{finite-dimensional reduction} 
of the full Feynman path integral on $M$.

Currently, there are two flavors of $\SL(2,\C)$ state integrals in the 
literature. The first, introduced in \cite{Hi,Hi2}, studied in \cite{DGLZ}, 
and made mathematically rigorous in \cite{AK}, is based on a 3-dimensional
lift of the 2-dimensional quantum Teichm\"uller theory in 
Kashaev's formalism \cite{Kashaev-Teich}. It uses variables associated to 
faces of tetrahedra.
The second, developed in \cite{D1}, explicitly uses shape parameters --- 
associated to edges of tetrahedra --- and constitutes a 3d lift of 
Teichm\"uller theory in the Fock-Chekhov formalism \cite{FC-Teich}. The 
two types of state integrals should be equivalent, though this has only 
been demonstrated in isolated examples so far \cite{SpirVar}.

It is the second state integral that we employ in this paper, due to its explicit dependence on shape parameters. Indeed, suppose that $M$ is an oriented one-cusped hyperbolic manifold with a $\rho_0$--regular triangulation $\CT$ and enhanced Neumann-Zagier datum $\widehat \beta_\CT=(z,\mb A,\mb B,f)$, with $\mb Af+\mb Bf''=\nu$. We must also assume that $\mb B$ is non-degenerate, which (Lemma \ref{lemma.B}) is always possible. Then we will%
\footnote{Here we multiply \eqref{appSIM} (at $u=0$) by an extra, canonical normalization factor $(2\pi/\hbar)^{3/2}$, in order to precisely match the asymptotics of the Kashaev invariant at the discrete faithful representation.} %
show in Appendix \ref{app.SIM} that the state integral of \cite{D1} takes the form
\begin{align} 
\lbl{SIM}
\CZ_{\CT}(\hbar) &= \sqrt{\frac{8\pi^3}{\hbar^3\det\mb B}}\int 
\frac{d^N\!Z}{(2\pi\hbar)^{N/2}}\,
e^{\tfrac1\hbar\big[ 
\tfrac1{2}\big(i\pi+\tfrac\hbar 2\big)^2f\cdot \mb B^{-1}\eta
-\big(i\pi+\tfrac\hbar2\big)Z\cdot\mb B^{-1}\eta
+\tfrac1{2}Z\cdot\mb{B}^{-1}\mb A Z\big]}
\prod_{i=1}^N \psi_{\hb}(Z_i)\,,  
\end{align}
where $\psi_\hbar(Z)$ is a non-compact quantum dilogarithm \cite{Ba, Fa}, the Chern-Simons partition function of a single tetrahedron. The integration variables $Z_i$ are, literally, the logarithmic shape parameters of $\CT$.

The integration contour of \eqref{SIM} is unspecified. A complete, non-perturbative definition of $\CZ_\CT(\hbar)$ requires a choice of contour, and the choice leading to invariance under 2--3 moves (etc.) may be quite subtle. However, a formal asymptotic expansion of the state integral as in \eqref{asympSIM} \emph{does not require} a choice of contour. It simply requires a choice of critical point for the integrand. Then the asymptotic series may be developed via \emph{formal Gaussian integration} in an infinitesimal neighborhood of the critical point.

We will show in Section \ref{sec.saddle} that all the leading-order critical points of \eqref{SIM} are logarithmic solutions to the gluing equations
\be \label{glueSIM} \text{critical points}\qquad \longleftrightarrow \qquad  z{}^{\mb A}(1-z^{-1}){}^{\mb B} = (-1)^\nu\,, \ee
with $z=\exp(Z)$. In particular, the critical points are isolated. Then, choosing the discrete faithful solution to \eqref{glueSIM}, we formally expand the state integral to find that
\begin{itemize}
\item $S_{\CT,0}$, the evaluation of leading-order part of the integrand at the critical point, is the 
complex volume of $M$;
\item $\exp(-2S_{\CT,1})$ is expressed as the determinant of a Hessian matrix
\be \notag
\CH = -\mb B^{-1}\mb A+\Delta_{1-z}^{-1}\,, 
\ee
with a suitable monomial correction, and reproduces the torsion 
\eqref{torintro}; and
\item the higher $S_{\CT,n}$ are obtained via a finite-dimensional Feynman 
calculus, and explicitly appear as rational functions of shape parameters.
\end{itemize}

It follows from the formalism of \cite{D1}, reviewed in Appendix \ref{app.SIM}, that the state integral \eqref{SIM} is only well defined up to multiplicative prefactors of the form
\be \label{abc}
\exp\Big( \frac{\pi^2}{6\hbar}a+\frac{i\pi}{4}b+\frac{\hbar}{24}c\Big)\,,
\qquad a,\,b,\,c\,\in\Z\,.
\ee
This means that we only obtain $\big(S_{\CT,0},\tau_\CT=4\pi^3e^{-2S_{\CT,1}},S_{\CT,2}\big)$ modulo $\big(\tfrac{\pi^2}{6}\Z,\,i,\,\tfrac{1}{24}\Z\big)$, respectively; however, all the higher invariants $S_{\CT,n\geq 3}$ should be unambiguous. Moreover, in Section \ref{sec.torsion} we saw that the ambiguity in $\tau_\CT$ could be lifted%
\footnote{It may also be possible to lift the ambiguities in $S_{\CT,0}$ and $S_{\CT,2}$ by using ordered triangulations, as in \cite{N} or \cite{Zickert-rep}.} %
to a sign $\pm 1$. Although the construction of the asymptotic series \eqref{asympSIM} appears to depend on $\CT$, we certainly expect that
\begin{conjecture}\label{conj.Sn}
The invariants $\big\{S_{\CT,n}\big\}_{n=0}^\infty$ are independent of the choice of regular triangulation and Neumann-Zagier datum (including the choice of quad type with $\det \mb B\neq 0$, etc.), up to the ambiguity \eqref{abc}, and thus constitute topological invariants of $M$.
\end{conjecture}

We now proceed to analyze the critical points and asymptotics of \eqref{SIM} in greater detail. In Section \ref{sec.uSIM}, we will also generalize the state integral to arbitrary representations, with non-unit meridian eigenvalue $m=e^u\neq 1$, and give an example of $S_{\CT,2}(m)$, $S_{\CT,3}(m)$ as functions on the character variety $Y_M^{\rm geom}$ for the figure-eight knot.

\subsection{Critical points}
\lbl{sec.saddle}

We begin by showing that the critical points of \eqref{SIM} are indeed 
solutions to the gluing equations. For this purpose, we need to know 
the quantum dilogarithm $\psi_{\hb}(Z)$. The latter is given by 
\cite[Eqn.3.22]{DGLZ}
%
\be \label{qdl} \psi_{\hb}(Z) = \prod_{r=1}^\infty \frac{1-q^r e^{-Z}} {1-({}^Lq)^{-r+1}e^{-{}^LZ}}\,,
\ee
for $|q|<1$, where
\be q := \exp \hbar,\,\quad {}^L q := \exp \frac{-4\pi^2}{\hbar}\,,
\qquad {}^LZ := \frac{2\pi i}{\hbar}Z\,.\ee
The quantum dilogarithm $\psi_{\hb}(Z)$ coincides with the restriction to
$|q|<1$ of Faddeev's quantum dilogarithm \cite{Fa}, as follows from
\cite[Eqn.3.23]{DGLZ}. $\psi_{\hb}(Z)$ is the Chern-Simons wavefunction 
of a single tetrahedron \cite{D1}.  
The quantum dilogarithm has an asymptotic expansion as $\hbar\to 0$, given by (\cf\ \cite[Eqn.3.26]{DGLZ})
\begin{align} \label{qdlasymp}
\psi_{\hb}(Z) &\overset{\hbar\to 0}{\sim} 
\exp \sum_{n=0}^\infty \frac{B_n\,\hbar^{n-1}}{n!} \wt \Li_{2-n}(e^{-Z}) \\
&= \exp\left[\frac{1}{\hbar}\wt\Li_2(e^{-Z})+\frac12\wt\Li_1(e^{-Z})
-\frac\hbar{12}z'+\frac{\hbar^3}{720}z(1+z)z'^3+\ldots\right]\,,  \notag
\end{align}
where $B_n$ is the $n^{\rm th}$ Bernoulli number, with $B_1=1/2$.

The coefficients of strictly positive powers of 
$\hbar$ (\ie\ $n\geq 2$) in the expansion are rational functions of 
$z=e^{Z}$, but the two leading asymptotics --- the logarithm and dilogarithm --- are multivalued and have branch cuts. 
In contrast, the function $\psi_{\hb}(Z)$ itself is a meromorphic function on $\C$ for 
any fixed $\hbar\neq 0$. Branch cuts in its asymptotics arise when 
families of poles collide in the $\hbar\to 0$ limit. In the case of purely imaginary $\hbar$ 
with $\mathrm{Im}\,\hbar>0$ (a natural choice in the analytic continuation 
of $\mathrm{SU}(2)$ Chern-Simons theory), a careful analysis of this 
pole-collision 
process leads to branch cuts for $\wt \Li_2$ and $\wt \Li_1$ that are 
different from the standard ones (Figure \ref{fig.cuts}). We indicate 
the modified analytic structure of these two functions (really functions 
of $Z$ rather than $e^{-Z}$) with an extra tilde.

\begin{figure}[htb]
\centering
\includegraphics[width=5.5in]{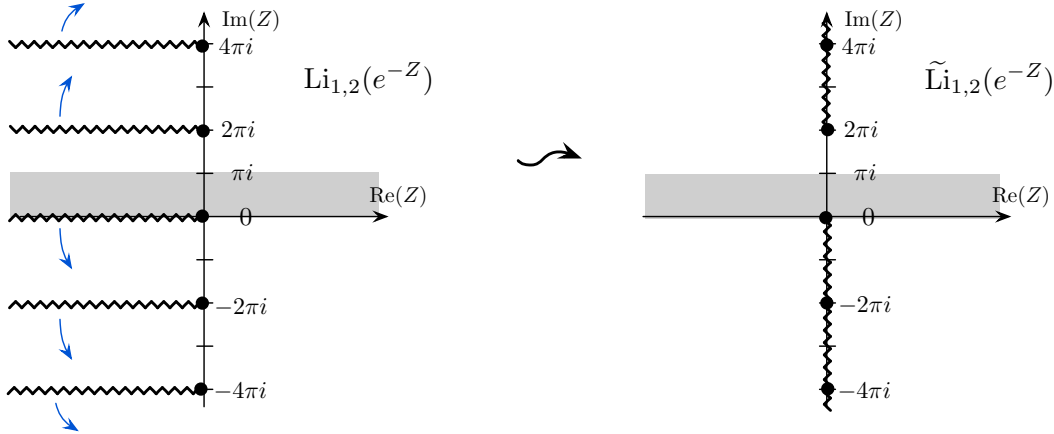}
\caption{Rotating the standard branch cuts of $\Li_2(e^{-Z})$ and 
$\Li_1(e^{-Z})$ to produce $\wt\Li_2(e^{-Z})$ and $\wt\Li_1(e^{-Z})$, as 
functions of $Z$. The shaded region indicates where the standard 
logarithms of shape parameters for the discrete faithful representation lie.}
\lbl{fig.cuts}
\end{figure}

Now, the critical points of the integrand, at leading order%
\footnote{We treat all sub-leading terms as perturbations. The exact 
location of the critical point will acquire perturbative corrections, 
described in Section \ref{sec.diags}.} %
in the $\hbar$ expansion, are solutions to
\begin{align} \notag
0 &= \frac{\partial}{\partial Z_i}\Big(-\frac{\pi^2}{2}f\cdot 
\mb B^{-1}\eta-i\pi Z\cdot\mb B^{-1}\eta+\frac1{2}Z\cdot\mb{B}^{-1}\mb A Z
+\sum_i\wt{\Li}_2(e^{-Z_i})\Big)  \\
&= -i\pi \big(\mb B^{-1}\eta)_i+\big(\mb B^{-1} \mb A Z)_i
-\wt \Li_1(e^{-Z_i})\,, \notag
\end{align}
in other words,
\be \mb A Z+\mb B (-\wt\Li_1(e^{-Z}))=i\pi \eta\,. \lbl{crit} 
\ee
Since $\exp[-\wt\Li_1(e^{-Z_i})]=1-z_i^{-1}$, we see that every 
solution to \eqref{crit} is a particular logarithmic lift of a solution 
to the actual gluing equations $z^{\mb A}(1-z^{-1})^{\mb B}=(-1)^\eta$. It is a 
lift that precisely satisfies the logarithmic constraints \eqref{NZlin} 
of Section \ref{sec.triang}, with $Z_i''=-\wt\Li_1(e^{-Z_i})$.

When $0\leq {\rm Im}\,Z_i\leq \pi$, the branches of the standard 
logarithms and dilogarithms agree with those of the modified ones. 
In particular, given the discrete faithful solution to 
$z^{\mb A}(1-z^{-1})^{\mb B}=(-1)^n$, taking standard logarithms 
immediately produces a solution to \eqref{crit}. 
Therefore, the discrete faithful representation always corresponds to a 
critical point of the state integral.

\subsection{Volume}
\lbl{sec.vol}

By substituting a solution to \eqref{crit} back into the $\hbar^{-1}$ 
(leading order) part of the integrand, we obtain the following formula for the complex 
volume of a representation: 
\be 
S_{\CT,0} = -\frac{\pi^2}{2}f\cdot \mb B^{-1}\eta
-i\pi Z\cdot\mb B^{-1}\eta+\frac1{2}Z\cdot\mb{B}^{-1}\mb A Z 
+\sum_i\wt{\Li}_2(e^{-Z_i}) \qquad \Big({\rm mod} \;\;\frac{\pi^2}{6}\Big)\,. 
\ee
Some manipulation involving the flattening can be used to 
recast this as 
\be \lbl{vol}
\boxed{ S_{\CT,0} = -\frac12 (Z-i\pi f)\cdot (Z''+i\pi f'')
+\sum_i\wt\Li_2(e^{-Z_i})} \qquad \Big({\rm mod} \;\;\frac{\pi^2}{6}\Big)\,,
\ee
where $Z_i'' := -\wt \Li_1(e^{-Z_i})$. It is straightforward 
to verify that this formula is independent of the choice of quad type,
choice of edge of $\CT$, choice of meridian loop, choice of
flattening, and 2--3 moves defines a topological invariant, which agrees
with the complex Chern-Simons invariant of $M$. Since the complex volume 
in this form has already been studied at length in the literature, 
we suppress the details here.

At the discrete faithful representation, we can remove the ``tildes''
from the logarithm and dilogarithm. If we consider the discrete faithful 
solution to $z^{\mb A}(1-z^{-1})^{\mb B}=(-1)^n$, and take standard logarithms 
$Z_i=\log z_i$, $Z_i''=\log(1-z_i^{-1})$ 
(with $0\leq {\rm Im}\,Z,\,{\rm Im}\,Z''\leq \pi$), we find
\begin{align}
S_{\CT,0} &= \,i({\rm Vol}(M)-i\,{\rm CS}(M)) \notag \\ &= 
-\frac12 (Z-i\pi f)\cdot (Z''+i\pi f'')+\sum_i\Li_2(e^{-Z_i}) 
\qquad \Big({\rm mod} \;\;\frac{\pi^2}{6}\Big)\,.
\end{align}
This is a version of the simple formula for the complex volume given in 
\cite{Neumann-combi}. It is known that the ambiguity in the volume can be 
lifted from $\pi^2/6$ to $2\pi^2$ using more refined methods 
\cite{N,DuZi,GoZi,Zickert-rep}.

\subsection{Torsion revisited}
\lbl{sec.torsion2}

Next, we can derive our torsion formula \eqref{torintro}. The torsion comes from the $\hbar^0$ part in the asymptotic expansion of the state integral,
which has several contributions.

From formal Gaussian integration around a critical point \eqref{crit}, 
we get a determinant $(2\pi \hbar)^{N/2}\big(\det \CH\big)^{-1/2}$, where
\begin{align} \CH_{ij} &= -\frac{\partial^2}{\partial Z_i
\partial Z_j}\Big(-\frac{\pi^2}{2}f\cdot \mb B^{-1}\eta-i\pi Z\cdot
\mb B^{-1}\eta+\frac1{2}Z\cdot\mb{B}^{-1}\mb A Z+\sum_i\wt{\Li}_2(e^{-Z_i})\Big) 
\notag \\
 &= \boxed{ (-\mb B^{-1}\mb A+\Delta_{z'})_{ij}} \lbl{Hess}
\end{align}
is the Hessian matrix of the exponent (at leading order $\hbar^{-1}$). 
Here $\Delta_{z'}:=\diag(z_1',...,z_N')$, with $z_i'=(1-z_i)^{-1}$ as usual. 
Multiplying the determinant is the $\hbar^0$ piece of the integrand, 
evaluated at the critical point. From the $\hbar^0$ part of the quadratic 
exponential, we get
\begin{align} \exp\Big(\frac{i\pi}{2}f\cdot \mb B^{-1}\eta
-\frac12 Z\cdot \mb B^{-1}\eta\Big)
 &= \exp\Big( \frac12 f\cdot (\mb B^{-1}\mb A Z+Z'')
-\frac12 Z\cdot (\mb B^{-1}\mb A f+f'')\Big) \notag \\
 &= \exp\Big(-\frac12 Z\cdot f''+\frac12 Z''\cdot f\Big) \notag \\
 & = \big(z^{f''}z''{}^{-f}\big)^{-1/2}\,, \lbl{torcorr1}
\end{align}
whereas from the quantum dilogarithm at order $\hbar^0$ we find
\be 
\exp\Big(\frac12\sum_i\wt\Li_1(e^{-Z_i})\Big) 
= \pm\prod_i \frac{1}{\sqrt{1-z_i^{-1}}}=\pm\det \Delta_{z''}^{-1/2}\,. 
\lbl{torcorr2} 
\ee
Combining the determinant $(2\pi \hbar)^{N/2}\big(\det \CH\big)^{-1/2}$, 
the corrections \eqref{torcorr1}--\eqref{torcorr2}, and the overall 
prefactor $\sqrt{\frac{8\pi^3}{\det \mb B}}(2\pi\hbar)^{-N/2}$ in the 
integral \eqref{SIM} itself, we finally obtain
\be
e^{S_1} = \sqrt{\frac{8\pi^3}{ \det\mb B \det(-\mb B^{-1}\mb A
+\Delta_{z'})\det\Delta_{z''}z^{f''}z''{}^{-f}}} = 
\sqrt{\frac{-8\pi^3}{\det(\mb A\Delta_{z''}
+\mb B\Delta_z^{-1})z^{f''}z''{}^{-f}}}\,, 
\ee
up to multiplication by a power of $i$; or
\be \tau_M := 4\pi^3e^{-2S_1} = \pm\frac12\det(\mb A\Delta_{z''}
+\mb B\Delta_z^{-1})z^{f''}z''{}^{-f}\,,\ee
just as in \eqref{torintro}. Despite the fact that the original 
state integral only made sense for non-degenerate $\mb B$, the final
formula for the torsion is well defined for any $\mb B$.

\subsection{Feynman diagrams and higher loops}
\lbl{sec.diags}

The remainder of the invariants $S_{\CT,n}$ can be obtained by continuing 
the saddle-point (stationary phase) expansion of the state integral to higher order. The 
calculation can be systematically organized into a set of Feynman rules 
(\cf\ \cite[Ch.9]{MirrorSym}, \cite{BIZ,Polyak}). The resulting formulas --- summarized in the Introduction --- are explicit 
algebraic functions of the exponentiated shape parameters $z_i$, and 
belong to the invariant trace field $E_M$.

To proceed, we should first re-center the integration around a 
critical point. Thus, we replace $Z\to Z+\zeta$ and 
integrate over $\zeta$, assuming $Z$ to be a solution to \eqref{crit}. 
Using \cite[Eqn.3.26]{DGLZ}, we expand as follows:
\begin{align} \notag
Z_{\CT}(\hbar) 
&=  \sqrt{\frac{8\pi^3}{\hbar^3\det\mb B}}\int 
\frac{d^N\! \zeta}{(2\pi\hbar)^{N/2}} \prod_{i=1}^N \psi_{\hb}(Z_i+\zeta_i) 
\\ &\hspace{.5in}\times 
e^{\tfrac1\hbar\big[ \tfrac1{2}\big(i\pi+\tfrac\hbar 2\big)^2f\cdot 
\mb B^{-1}\eta-\big(i\pi+\tfrac\hbar2\big)(Z+\zeta)\cdot\mb B^{-1}\eta
+\tfrac1{2}(Z+\zeta)\cdot\mb{B}^{-1}\mb A (Z+\zeta)\big]}
 \notag\\
&\sim \sqrt{\frac{8\pi^3}{\hbar^3\det\mb B}}\,e^{\Gamma^{(0)}(Z)}\int 
\frac{d^N\! \zeta}{(2\pi\hbar)^{N/2}}\exp\Big[-\frac{1}{2\hbar}
\zeta\cdot\CH(Z)\cdot \zeta+\sum_{k=1}^\infty\sum_{i=1}^N 
\frac{\Gamma^{(k)}_i(Z)}{k!}\,\zeta_i^k\Big].
\label{Feyngen}
\end{align}
In this form, the first coefficient $\Gamma^{(0)}(Z)$ can be identified with an overall \emph{vacuum energy}, while the rest of the $\Gamma_i^{(k)}(Z)$ are \emph{vertex factors}.

Every $\Gamma^{(k)}(Z)$ here is a series in $\hbar$, in general 
starting with a $1/\hbar$ term. However, $\Gamma_i^{(1)}$ must vanish 
at leading order $\hbar^{-1}$ precisely because $Z$ is a solution to 
the leading-order critical point equations; and we have also already 
extracted the leading $\hbar^{-1}$ piece of $\Gamma_i^{(2)}$ as the Gaussian 
integration measure $-\tfrac{1}{2\hbar}\zeta\CH\zeta$.
Typically, 1-vertices and 2-vertices are absent from a 
Feynman calculus. Here, however, they appear because our critical 
point equation and the Hessian (respectively) are only accurate at 
leading order, and incur $\hbar$-corrections. (Note that the 1-vertices and 2-vertices are counted separately in \eqref{effloop} below.)

The vacuum energy $\Gamma^{(0)}$ contributes to every $S_{\CT,n}$, $n\geq 0$. Its leading-order $\hbar^{-1}$ term is just the complex volume \eqref{vol}, while the $\hbar^0$ piece contains 
the corrections \eqref{torcorr1}--\eqref{torcorr2} to the torsion. At 
higher order in $\hbar$, we have
\be \Gamma^{(0)}(Z) = \frac{1}{\hbar}S_0+\hbar^0(...)
+\frac{\hbar}{8}f\cdot\mb B^{-1}\mb A f + \sum_{n=2}^\infty 
\frac{\hbar^{n-1} B_{n}}{n!}\sum_{i=1}^N \Li_{2-n}(z_i^{-1})\quad\; 
\Big({\rm mod}\;\frac{\hbar}{24}\Big)\,. \ee

Each $S_n$, $n\geq 2$, is calculated by taking the $\hbar^{n-1}$ part 
of $\Gamma^{(0)}$, and adding to it an appropriate sum of Feynman diagrams. 
The rules for the diagrams are derived from \eqref{Feyngen} as follows. There are vertices of all 
valencies $k=1,2,...$, with a vertex factor given by $\Gamma^{(k)}_i$. One draws all connected diagrams (graphs) with
\be \lbl{effloop}
\boxed{\text{\# loops + \# 1-vertices + \# 2-vertices $\leq n$}\,.}
\ee
Each $k$-vertex is assigned a factor $\Gamma_i^{(k)}$, and each edge is assigned a propagator
\be 
\lbl{prop}
\text{propagator}:\qquad \Pi_{ij}:= \hbar\,\CH^{-1}_{ij} \,
=\, \hbar(-\mb B^{-1}\mb A+\Delta_{z'})^{-1}_{ij}\,.
\ee
The diagrams are then evaluated by contracting the vertex factors with propagators, and multiplying by a standard \emph{symmetry factor}. In each diagram, one should restrict to the $\hbar^{n-1}$ term in its evaluation.

Explicitly, using the asymptotic expansion \eqref{qdlasymp} of the quantum dilogarithm, we find that the vertices are 
\begin{subequations} 
\lbl{vx}
\begin{align}
\text{1-vertex:}\quad \Gamma^{(1)}_i &= -\frac12(\mb B^{-1}\eta)_i
-\sum_{n=1}^\infty \frac{\hbar^{n-1}B_n}{n!}\Li_{1-n}(z_i^{-1}) 
= -\frac12(\mb B^{-1}\eta)_i  +\frac{z'_i}{2}+\ldots\,,\\
\text{2-vertex:}\quad \Gamma^{(2)}_i &= \sum_{n=1}^\infty 
\frac{\hbar^{n-1}B_n}{n!}\Li_{-n}(z_i^{-1}) = \frac{z_iz_i'^2}{2}
-\frac\hbar{12}z_i(1+z_i)z_i'^3+\ldots\\
\text{$k$-vertex:}\quad\Gamma^{(k)}_i &= (-1)^k \sum_{n=0}^\infty 
\frac{\hbar^{n-1}B_n}{n!}\Li_{2-n-k}(z_i^{-1}) \qquad\qquad (k\geq 3)
\end{align}
\end{subequations}
Note that in $\Gamma_i^{(1)}$ we could also write 
$\mb B^{-1}\eta=\mb B^{-1}\mb Af+f''$.
When the inequality \eqref{effloop} is saturated, only the leading-order ($\hbar^{-1}$ or $\hbar^0$) 
terms of the vertex factors \eqref{vx} need be considered. Otherwise, subleading $\hbar$-corrections may be necessary.

Examples of 2-loop and 3-loop Feynman diagrams were given in Figures \ref{fig.diags2intro}--\ref{fig.diags3dot} of the Introduction, along with the entire evaluated expression for $S_{\CT,2}$.

\subsection{$n$-loop invariants on the character variety}
\label{sec.uSIM}

Just as we extended the torsion formula to general representations $\rho\in X_M$ in Section \ref{sec.extensions}, we may now generalize the entire state integral. The basic result for the higher invariants $S_{\CT,n}$ is that their formulas remain completely unchanged. The shapes $z_i$ simply become functions of the representation $\rho$, and satisfy deformed gluing equations \eqref{uglue}--\eqref{uglueL}. One must also make sure to use a generalized flattening whenever it occurs, just as in Section \ref{sec.extensions}.

We note that, for a hyperbolic knot complement $M=S^3\backslash K$, the generalized Chern-Simons state integral $\CZ_M(u;\hbar)$ is expected to match the asymptotic expansion of the colored Jones polynomials $J_N(K;q)$. Specifically, one should consider the limit
\be N\to \infty\,,\quad \hbar\to 0\,,\qquad q^N=e^{N\hbar}=e^{2u}\quad\text{fixed}\,,\ee
where $m=e^u$ is the meridian eigenvalue for a geometric representation $\rho_m$ in the neighborhood of the discrete faithful. This is the full \emph{Generalized Volume Conjecture} of \cite{Gu}.

To see how formulas for the generalized invariants $S_{\CT,n}$, $n\geq 0$, come about, consider the state integral at general meridian eigenvalue $m=e^u$. From \eqref{appSIM} of Appendix \ref{app.SIM}, we find
\begin{align} \notag
\CZ_\CT(u;\hbar) 
&=  \sqrt{\frac{8\pi}{\hbar^3\det\mb B}}\int \frac{d^N\!Z}{(2\pi \hbar)^{N/2}} \, \prod_{i=1}^N \psi_\hbar(Z_i)\,e^{-\tfrac1{2\hbar}Z\cdot \mb B^{-1}\mb A Z}\\
 &\hspace{1in} \times
e^{\tfrac1\hbar\Big[
2 \bm u\cdot \mb D\mb B^{-1}\bm u+(2\pi i+\hbar)f\cdot 
\mb B^{-1}\bm u
+\tfrac12\big(i\pi+\tfrac\hbar2\big)^2 f\cdot \mb B^{-1} \eta 
- Z\cdot\mb B^{-1}\big(2\bm u+\big(i\pi+\tfrac\hbar2\big)\eta\big)\Big]}    \,,
 \label{uSIM}
\end{align}
where $\bm u:=(0,...,0,u)$ and $\mb D$ is the block appearing in any completion of the Neumann-Zagier matrices $(\mb A\;\mb B)$ to $\left(\begin{smallmatrix}\mb A&\mb B\\\mb C&\mb D\end{smallmatrix}\right)\in Sp(2N,\mathbb{Q})$, such that the bottom row $D$ of $\mb D$ appears in the longitude gluing equation $C\cdot Z+D\cdot Z''=v+2\pi i\nu_\lambda$ (Section \ref{sec.uflat}). Indeed, since we are contracting with $\bm u$, only this bottom row of $\mb D$ really matters in \eqref{uSIM}.

The critical points of the state integral are now given by
\be \label{ucrit} \mb A Z+\mb B Z''=2\bm u+i\pi\nu\,,\ee
with $Z'':=-\wt\Li_1(e^{-Z})$. As expected, this is the logarithmic form of the deformed gluing equation \eqref{uglue}. Thus, all critical points correspond to representations $\rho=\rho_m\in X_M$. The multivalued nature of this equation must be carefully studied to make sure desired solutions actually exist. However, for example, representations on the geometric component $X_M^{\rm geom}$ always exist in a neighborhood of the discrete faithful representation, if we choose $u=\log m$ to be close to zero (and use a regular triangulation).

We then start expanding the state integral around a critical point, setting
\be \CZ_{\CT}(u;\hbar) \sim \hbar^{-\frac32}\exp\Big[\frac1\hbar S_{\CT,0}(u)+S_{\CT,1}(u)+\hbar S_{\CT,2}(m)+\hbar^2 S_{\CT,3}(m)+...\Big]\,.
\ee
The leading contribution $S_{\CT,0}(\rho)$ is given, following some standard manipulations using the generalized flattening, by
\be \label{S0u}
S_{\CT,0}(u) =   u\,v(u)-\frac12 (Z-i\pi f)\cdot (Z''+i\pi f'') 
+ \sum_{i=1}^N\wt\Li_2(e^{-Z})\qquad \Big(\text{mod}\;\frac{\pi^2}{6}\Big)\,.
\ee
Here we write $S_{\CT,0}$ as a function of the logarithmic meridian eigenvalue $u$, though a fixed choice of representation $\rho$ will implicitly fix the choice of longitude eigenvalue $v=\log(-\ell)$ as well. Expression \eqref{S0u} is a holomorphic version of the complex volume of a cusped manifold with deformed cusp. Explicitly,
\be S_{\CT,0}(u) = i({\rm Vol}_M(u)+i{\rm CS}_M(u)) -2v\,\Re(u)\,.\ee
This is the correct form of the complex volume to use in the Generalized Volume Conjecture, \cf\ \cite{GuM}.

At first subleading order, we re-derive the generalized torsion formula. The calculation is identical to that of Section \ref{sec.torsion2}, with the exception of the correction \eqref{torcorr1} coming from the $\hbar^0$ part of the exponential. This correction now becomes
\be \exp\Big[ f\cdot \mb B^{-1}\bm u+\frac{i\pi}{2}f\cdot \mb B^{-1}\nu-\frac12Z\cdot \mb B^{-1}\nu\Big]\,. \label{utorcorr1}\ee
To simplify this correction, we must use $\mb Af+\mb Bf''=\nu$ and a deformed gluing equation $\mb A Z+\mb B Z''=2\bm u+i\pi \nu$. The $\bm u$-dependent part of the gluing equation cancels the new $\bm u$-dependent term in \eqref{utorcorr1}, ultimately leading to the same result
\be \exp\Big[ f\cdot \mb B^{-1}\bm u+\frac{i\pi}{2}f\cdot \mb B^{-1}\nu-\frac12Z\cdot \mb B^{-1}\nu\Big] = \big(z{}^{f''}z''{}^{-f}\big)^{-1/2}\,, \notag
\ee
and therefore the same torsion%
\footnote{The normalization of the torsion here differs from the torsion at the discrete faithful by a factor of $\pi^2$. In fact, we intentionally changed the normalization of the entire state integral \eqref{uSIM} by $\pi^2$. This is because we wanted the state integral to match the asymptotics of the colored Jones polynomials exactly, and the asymptotics happen to jump by $\pi^2$ when $u\neq 0$, \cf\ \cite{GuM}.}
\be \tau_\CT = 4\pi e^{-2S_{\CT,1}} = \frac12\det\big(\mb A\Delta_{z''}
+\mb B\Delta_z^{-1}\big)z^{f''}z''{}^{-f}\,.\ee

Finally, we can produce a generalized version of the Feynman rules of Section \ref{sec.diags}. We note, however, that the $u$-dependent terms in \eqref{uSIM} do not contribute to either the vacuum energy $\Gamma^{(0)}$ (at order $\hbar^1$ or higher), the propagator, or the vertex factors $\Gamma^{(k)}_i$. Therefore, the Feynman rules must look exactly the same. The only difference is that the critical point equation \eqref{ucrit} requires us to use shape parameters that satisfy the generalized gluing equations.

\subsection{Example: $\mb{4_1}$ completed}
\label{sec.41Sn}

We may demonstrate the power of the Feynman-diagram approach by computing the first two subleading corrections $S_{\CT,2}$ and $S_{\CT,3}$ for the figure-eight knot complement.

We can use the same Neumann-Zagier datum described in Section \ref{sec.41triang}, along with the generalized flattening of Section \ref{sec.utor41}. Let us specialize to representations $\rho_m$ on the geometric component of the character variety. Then the two shapes $z,\,w$ are expressed as functions on the A-polynomial curve,
\be \label{zw412} z = -\frac{m^2-m^{-2}}{1+m^2\ell} \,,\qquad w= \frac{m^2+\ell}{m^2-m^{-2}}\,,\ee
as in \eqref{zw41m}.

The 2-loop invariant is explicitly given in \eqref{2loopexplicit} of the Introduction. Evaluating this expression in \texttt{Mathematica}, we find
\be\notag S_{\mb{4_1},2}=-\frac{w^3 (z+1)+w^2 ((11-8 z) z-4)+w (z-1) (z (z+12)-5)+(z-2) (z-1)^2}{12 (w+z-1)^3}\,. \ee
Upon using \eqref{zw412} to substitute rational functions for $z$ and $w$, the answer may be most simply expressed as
\be \wt S_{\mb {4_1},2}= \frac{S_{\mb {4_1},2}+1/8}{\tau_{\mb {4_1}}^3} = -\frac{1}{192}\big(m^{-6}-m^{-4}-2m^{-2}+15-2m^2-m^4+m^6\big)\,,\ee
where we have divided by a power of the torsion as suggested in \eqref{tordivide} of the Introduction. (We have also absorbed a constant $1/8$, recalling that our formula is only well defined modulo $\BZ/24$.

In a similar way, we may calculate the 3-loop invariant, finding unambiguously
\be \wt S_{\mb{4_1},3} = \frac{S_{\mb{4_1},3}}{\tau_{\mb{4_1}}^6}=\frac{1}{128}\big(m^{-6}-m^{-4}-2m^{-2}+5-2m^2-m^4+m^6\big)\,.\ee

These answers agree perfectly with the findings of \cite{DGLZ}, and the comparison there to the asymptotics of the colored Jones polynomials at general $u$. Moreover, at the discrete faithful representation we obtain
\be S_{\mb {4_1},2} = \frac{11i}{72\sqrt{3}}= -\frac{11}{192\,\tau_{\mb{4_1}}^{3}}\,,\qquad S_{\mb{4_1},3}=-\frac1{54} = \frac{1}{128\,\tau_{\mb{4_1}}^{6}}\,,\ee
in agreement with known asymptotics of the Kashaev invariant.

\appendix


\section{Symplectic properties of A and B}
\lbl{app.symp}

The $N\times N$ Neumann-Zagier matrices $\mb A$ and $\mb B$ form the top 
half of a symplectic matrix 
$\left(\begin{smallmatrix} \mb A & \mb B  \\ \mb C & \mb D 
\end{smallmatrix}\right) \in \Sp(2N,\mathbb Q)$ \cite{NZ}. In this section
we discuss some elementary properties of symplectic matrices.

\begin{lemma}
\lbl{lem.halfAB}
The $N \times 2N$ matrix $(\mb A \; \mb B)$ is the upper half of a 
symplectic matrix if and only if 
$\mb A \mb B^T$ is symmetric and $(\mb A \; \mb B)$ has maximal rank $N$.
\end{lemma}

\begin{proof}
It is easy to see that the rows of $(\mb A \; \mb B)$ have zero symplectic 
product (with respect to the standard symplectic form on $\mathbb{Q}^{2N}$
if and only if $\mb A \mb B^T$ is symmetric. In addition they span a vector
space of rank $N$ if and only if $(\mb A \, \mb B)$ has maximal rank $N$.
The result follows.
\end{proof}

\begin{lemma}
\lbl{lemma.BA}
If $(\mb A\;\mb B)$ is the upper half of a symplectic matrix and 
$\mb B$ is non-degenerate, then $\mb B^{-1}\mb A$ is symmetric.
\end{lemma}

\begin{proof}
Lemma \ref{lem.halfAB} implies that $\mb A \mb B^T$ is symmetric, and 
so is $(\mb B)^{-1} \mb A \mb B^T  ((\mb B)^{-1})^T$.
\end{proof}  
  
It is not true in general that $\mb B$ is invertible. However, after a 
possible change of quad type, we can assume that $\mb B$ is invertible. 
This is the content of the next lemma.

\begin{lemma}
\lbl{lemma.B}
\rm{(a)} Suppose $(\mb A \; \mb B)$ is the upper half of a symplectic
$2N \times 2N$ matrix. If $\mb A$ has rank $r$, then any $r$ linearly
independent columns of $\mb A$
and their complementary $N-r$ columns in $\mb B$ form a basis for the column
space of $(\mb A \; \mb B)$.
\newline
\rm{(b)}
There always exists a choice of quad type for which $\mb B$ is 
non-degenerate (for any fixed choice of redundant edge and meridian path).
\end{lemma}

\begin{proof}
For (a) let $\rank(\mb A)=r\leq N$. Without loss of generality, 
we may suppose that the first $r$ columns of $\mb A$ are linearly independent. 
We want to show that, together with the last $N-r$ columns of $\mb B$, 
they form a matrix of rank $N$.

If we simultaneously multiply both $\mb A$ and $\mb B$ on the left by 
any nonsingular matrix $U\in \mathrm{GL}(N,\mathbb R)$, both the symplectic 
condition and the columns are preserved. This follows from the fact 
that 
$\left(\begin{smallmatrix} U & 0 \\ 0 & U^{-1,T} \end{smallmatrix}\right)\in 
\Sp(2N,\mathbb R)$. By allowing such a transformation, we may 
assume that  $\mb A$ takes the block form
\be 
\label{Ablock} 
\mb A = \begin{pmatrix} \mb I_{r\times r} & A_2 \\ 0 & 0 \end{pmatrix} 
\ee
for some $A_2$. Similarly, we split $\mb B$ into blocks of size $r$ and 
$N-r$,
\be 
\mb B = \begin{pmatrix} B_1 & B_2 \\ B_3 & B_4 \end{pmatrix}\,. 
\ee

Since $(\mb A\;\mb B)$ has full (row) rank, we see 
that the bottom $N-r$ rows of $\mb B$ must be linearly independent, \ie\ 
$\rank(B_3\;\;B_4)=N-r$. From the symplectic condition of Lemma 
\ref{lem.halfAB}, we also find that $B_3+B_4A_2^T = 0$, so that 
$\rank (B_3\;\; B_4) \leq \rank(B_4)$. This then implies that $B_4$ 
itself must have maximal rank $N-r$. Therefore, the last $N-r$ columns 
of $\mb B$ are linearly independent, and also independent of the first 
$N$ columns of $\mb A$; \ie\ the matrix $\left(\begin{smallmatrix} 
\mb I_{r\times r} & B_2 \\ 0 & B_4 \end{smallmatrix}\right)$ has maximal 
rank as desired. This concludes the proof of part (a).

For part (b) let us denote the columns of $\mb A$ and $\mb B$ as 
$a_i$ and $b_i$. A change of quad type corresponding to a cyclic
permutation $Z_i\mapsto Z_i'\mapsto Z_i'' \mapsto Z_i$ on 
the $i^{\rm th}$ tetrahedron permutes the $i^{\rm th}$ columns of 
$\mb A$ and $\mb B$ as $(a_i,b_i)\mapsto (b_i-a_i, -a_i)$. Therefore, 
given $N$ complementary columns of $(\mb A\;\mb B)$ that have full rank, 
we can use such permutations to move all the columns (up to a sign) into 
$\mb B$.
\end{proof}


\section{The shape parameters are rational functions on the 
character variety}
\lbl{app.shapes}

In this appendix, we prove that the shape parameters of a regular 
ideal triangulation are rational functions on $Y_M^{\rm geom}$, 
the geometric component of the $\SL_2(\BC)$ A-polynomial curve.

\begin{proposition}
\lbl{prop.shapes}
Fix a regular ideal triangulation $\CT$ of a one-cusped hyperbolic manifold
$M$. Then every shape parameter of $\CT$ is a rational function on 
$Y_M^{\rm geom}$.
\end{proposition}

\begin{proof}
The proof is a little technical, and follows from work of  
Dunfield \cite[Cor.3.2]{Dn}, partially presented in the appendix to 
\cite{BDR-V}. For completeness, we give the details of the proof here.
We thank N. Dunfield for a careful explanation of his proof to us.

Consider the affine variety $R(M,\SL(2,\BC))=\Hom(\pi_1,\SL(2,\BC))$
and its algebrogeometric quotient $X_{M,\PSL(2,\BC)}$ 
by the conjugation action of $\PSL(2,\BC)$. Following Dunfield from the
Appendix to \cite{BDR-V}, let $\overline{R}(M,\SL(2,\BC))$ denote
the subvariety of $R(M,\SL(2,\BC)) \times P^1(\BC)$ consisting of pairs 
$(\rho,z)$ where $z$ is a fixed point of $\rho(\pi_1(\pt M))$. Let 
$\overline{X}_{M,\SL(2,\BC)}$ denote the algebro-geometric quotient of 
$\overline{R}(M,\SL(2,\BC))$ under the diagonal action of $\SL(2,\BC)$
by conjugation and M\"obius transformations respectively. 
We will call elements $(\rho,z) \in \overline{R}(M,\SL(2,\BC))$ 
{\em augmented representations}. Their images in the augmented character 
variety $\overline{X}(M,\SL(2,\BC))$  will be called 
{\em augmented characters} and will be denoted by square
brackets $[(\rho,z)]$. Likewise, replacing $\SL(2,\BC)$ by
$\PSL(2,\BC)$, we can define the character variety $X_{M,\PSL(2,C)}$ and
its augmented version $\overline{X}_{M,\PSL(2,\BC)}$.

The advantage of the augmented character variety $\overline{X}_{M,\SL(2,\BC)}$
is that given $\ga \in \pi_1(\pt M)$ there is a regular function $e_{\ga}$ 
that sends $[(\rho,z)]$ to the eigenvalue of $\rho(\ga)$ corresponding to 
$z$, using Lemma \ref{lem.eigen} below.
In contrast, in $X_{M,\SL(2,\BC)}$ only the trace $e_{\ga}+e_{\ga}^{-1}$
of $\rho(\ga)$ is well-defined. 
Likewise, in $\overline{X}_{M,\PSL(2,C)}$ (resp. $X_{M,\SL(2,\BC)}$) only 
$e_{\ga}^2$ (resp. $e_{\ga}^2+e_{\ga}^{-2}$) is well-defined.

From now on, we will restrict to the geometric component of the character
variety $X_{M,\PSL(2,\BC)}$ and we will fix a regular ideal triangulation
$\CT$. In \cite[Thm.3.1]{Dn} Dunfield proves that the natural restriction
map
$$
X_{M,\PSL(2,\BC)} \longto X_{\pt M,\PSL(2,\BC)}
$$
of affine curves is of degree $1$. 
$X_{\pt M,\PSL(2,\BC)}$ is an affine curve in $(\BC^*)^2/\BZ_2$
and let $V_{M,\PSL(2,\BC)} \subset (\BC^*)^2$ denote the preimage of 
$X_{\pt M,\PSL(2,\BC)}$ of the 2:1 map $(\BC^*)^2 \longto (\BC^*)^2/\BZ_2$. 
The commutative diagram
\begin{equation}
\label{eqn:diagram}
\cxymatrix{{{\overline{X}_{M,\PSL(2,\BC)} }\ar[r]\ar[d]
&V_{M,\PSL(2,\BC)}
\ar[d]
\\
{X_{M,\PSL(2,\BC)}}\ar[r]
&{X_{\pt M,\PSL(2,\BC)}}}}
\end{equation}
has both vertical maps of degree 2, and the bottom horizontal map of
degree 1. Thus, it follows that the top horizontal map is of degree 1.
In \cite[Sec.10.3]{BDR-V} Dunfield constructs a degree 1 developing map
$$
V_{\CT} \longto \overline{X}_{M,\PSL(2,\BC)} \,,
$$
which combined with the previous discussion gives a chain of birational
curve isomorphisms
\be
\lbl{eq.3curves}
V_{\CT} \longto \overline{X}_{M,\PSL(2,\BC)} \longto V_{M,\PSL(2,\BC)}
\ee
Since the shape parameters are rational (in fact coordinate) functions on 
$V_{\CT}$, it follows that they are rational functions on $V_{M,\PSL(2,\BC)}$.
Using the regular map $V_{M,\SL(2,\BC)} \longto V_{M,\PSL(2,\BC)}$, we obtain
that the shape parameters are rational functions on $V_{M,\SL(2,\BC)}$.
\end{proof}

Proposition \ref{prop.shapes} has the following concrete corollary.

\begin{corollary}
\lbl{cor.zML}
Given a regular ideal triangulation $\CT$ with $N$ tetrahedra, there is
a solution of the shape parameters in $\BQ(m,\ell)/(A(m,\ell))$.
\end{corollary}

\begin{lemma}
\lbl{lem.eigen}
Suppose $A=\left(\begin{matrix}a & b \\ c & d \end{matrix}\right) \in 
\SL(2,\BC)$ and $c \neq 0$. Then, $\lambda$ is an eigenvalue of $A$ if and 
only if $z=(\lambda-2d)/(2c)$ is a fixed point of the corresponding M\"obius
transformation in $P^1(\BC)$.
\end{lemma}


\section{Deriving the state integral}
\lbl{app.SIM}

In this appendix, we explain the connection between the quantization 
formalism of \cite{D1} and the special state integrals \eqref{SIM} and 
\eqref{uSIM} that led to all the formulas in the present paper. We will 
first review classical ``symplectic gluing'' of tetrahedra, then extend 
gluing to the quantum setting and construct the state integral. There are 
multiple points in the construction that have yet to be made mathematically 
rigorous, which we will try to indicate.

\subsection{Symplectic gluing}
\label{sec.sympglue}

The main idea of \cite{D1} is that gluing of tetrahedra should be viewed, both classically and quantum mechanically, as a process of symplectic reduction.

Suppose we have a one-cusped manifold $M$ with a triangulation $\CT=\{\Delta_i\}_{i=1}^N$. Classically, each tetrahedron $\Delta_i$ comes with a phase space
\begin{align} \CP_{\pd\Delta_i} &\;=\;\{\text{flat $\SL(2,\C)$ connections on $\pd\Delta_i$}\} \notag\\ &\;\approx\; \{(Z_i,Z_i',Z_i'')\in \C\backslash(2\pi i\Z)\,|\,Z_i+Z_i'+Z_i''=i\pi\}\,, \end{align}
with (holomorphic) symplectic structure
\be \omega_{\pd \Delta_i} = dZ\wedge dZ''\,,\ee
and a Lagrangian submanifold%
\footnote{Explicitly, $\CP_{\pd\Delta_i}$ is a space of flat connections on a 4-punctures sphere with parabolic holonomy at the four punctures; while $\CL_{\Delta_i}$ is the subspace with trivial holonomy --- hence connections that extend into the bulk of the tetrahedron. See, \eg, Section 2 of \cite{DGG1}.}
\begin{align} \label{Lagtet} \CL_{\Delta_i} &\;=\; \{\text{flat $\SL(2,\C)$ connections that extend to $\Delta_i$}\} \notag\\ &\;=\; \{e^{Z''}+e^{-Z}-1=0\}\;\subset \CP_{\pd\Delta_i} \end{align}

When gluing the tetrahedra together, we first form a product
\be \CL_{\times} = \CL_{\Delta_1}\times\cdots\CL_{\Delta_N}\quad\subset \quad \CP_{\times} = \CP_{\pd \Delta_1}\times\cdots\times \CP_{\pd \Delta_N}\,.\ee
The edge constraints $X_I:= \sum_{i=1}^N \big(\mb{G}_{Ii}Z_i+\mb{G}_{Ii}'Z_i' +\mb{G}_{Ii}''Z_i''\big)-2\pi i$ from \eqref{glue} are functions on the product phase space $\CP_\times$, and can be used as (holomorphic) moment maps to generate $N-1$ independent translation actions $t_I$. Recall \cite{NZ} that the logarithmic meridian and longitude holonomies $(u,v)$ are also functions on $\CP_\times$, which Poisson-commute with all the edges $X_I$, and so are fixed under these translations. Then the phase space of $M$ is a symplectic quotient,
\begin{align} \CP_{\pd M} &\;=\; \{\text{flat $\SL(2,\C)$ connections on $\pd M\simeq T^2$}\} \;\approx\; \{(u,v)\in \C\} \notag \\
 &\;=\; \CP_\times \big/\!\!\big/(t_I)\,,
\end{align}
and the A-polynomial of $M$ (more properly, components of the A-polynomial for which the triangulation is regular) is the result of pulling the Lagrangian $\CL_\times$ through the quotient,
\be \CL_M = \text{``}\CL_\times\big/\!\!\big/(t_I)\text{''} \quad \approx\; \{A_M(e^v,e^u)=0\}\;\subset\; \CP_{\pd M}\,.\ee
This is quite easy to check using equations \eqref{uglue} and \eqref{uglueL}.

\subsection{Quantization}

Quantum mechanically, each tetrahedron has a Hilbert space $\CH_{\Delta_i}$, a wavefunction $\CZ_{\Delta_i}(Z_i)$, and a quantum operator $\hat\CL_{\pd\Delta_i}$ that annihilates the wavefunction. The symplectic-gluing procedure extends to the quantum setting, with appropriate quantum generalizations of all the above operations. Roughly, one forms a product wavefunction
\be \label{prodH}
\CZ_\times(Z_1,...,Z_N)=\CZ_{\Delta_1}\otimes \cdots \otimes \CZ_{\Delta_N}\;\;\in\;\; \CH_\times = \CH_{\pd\Delta_1}\otimes\cdots \otimes\CH_{\pd\Delta_N}\,,
\ee
and restricts the product Hilbert space using $N-1$ new polarizations coming from the edge constraints. The resulting restricted wavefunction is $\CZ_M(u)$, and it is annihilated by a quantized version of the A-polynomial \cite{Gu,Ga-Ahat}.

To make this more precise, let $M$ again be an oriented one-cusped manifold, and choose a triangulation $\CT =\{\Delta_i\}_{i=1}^N$ (regular with respect to some desired family of representations), a quad type, a redundant edge, and a meridian path --- just as in Section \ref{sec.triang}.

To each tetrahedron $\Delta_i$ we associate a boundary Hilbert space $\CH_{\pd \Delta_i}$. It is some extension%
\footnote{This space has not been mathematically defined yet; constructions of (\eg) \cite{AK} might prove useful achieving this.} %
of $L^2(\BR)$ that includes the wavefunction
\be \CZ_{\Delta_i}(Z_i;\hbar):= \psi_\hbar(Z_i)\,,\ee
where $\psi_\hbar(Z_i)$ is Faddeev's quantum dilogarithm \eqref{qdl} \cite{Fa}. We also associate to $\Delta_i$ an algebra of operators
\be \label{algi} \hat{\mathcal{A}}_{\pd \Delta_i} = \C\la \hat Z_i,\hat Z_i',\hat Z_i''\ra/(\hat Z_i+\hat Z_i'+\hat Z_i''=i\pi+\tfrac\hbar2)\,,\ee
with commutation relations
\be \label{Zcomm} [\hat Z_i,\hat Z_i']=[\hat Z_i',\hat Z_i'']=[\hat Z_i'',\hat Z_i]=\hbar\,.\ee
Then the quantization of the Lagrangian \eqref{Lagtet} annihilates the wavefunction,
\be \hat \CL_{\Delta_i}\;:=\;e^{\hat Z_i''}+e^{-\hat Z_i}-1\,,\qquad \hat \CL_{\Delta_i}\CZ_{\Delta_i}=0\,, \ee
where the operators act in the representation
\be \hat Z_i=Z_i\,,\quad \hat Z_i''=\hbar \pd_{Z_i}\,;\quad\text{or}\quad
 e^{\hat Z_i}\CZ(Z_i)=e^{Z_i}\CZ(Z_i)\,,\quad e^{\hat Z_i''}\CZ(Z_i) = \CZ(Z_i+\hbar)\,.\ee

In order to glue the tetrahedra together, we start by forming the product wavefunction $\CZ_\times(Z_1,...,Z_N) = \CZ_{\Delta_1}(Z_1)\cdots\CZ_{\Delta_N}(Z_N)$. This is an element of a product Hilbert space \eqref{prodH}. Acting on this product Hilbert space is the product $\hat{\mathcal{A}}_\times$ of algebras \eqref{algi}, which is simply generated by all the $\hat Z_i,\hat Z_i',\hat Z_i''$, with canonical commutation relations \eqref{Zcomm} (and operators from distinct tetrahedra always commuting).

Now, following the notation of Sections \ref{sec.glue} and \ref{sec.uflat}, we can define $N$ operators $\hat X_I \in \hat{\mathcal{A}}_\times$, one for each independent edge, and one for the meridian:
\be \hat X_I :=  \begin{cases} \sum_{i=1}^N \big(\mb{G}_{Ii}\hat Z_i+\mb{G}_{Ii}'\hat Z_i'+\mb{G}_{Ii}''\hat Z_i''\big) - 2\pi i-\hbar &\;\; I=1,...,N-1\,, \\[.2cm]
 \mb{G}_{N+1,i}\hat Z_i+\mb{G}_{N+1,i}'\hat Z_i'+\mb{G}_{N+1,i}''\hat Z_i'' &\;\; I=N\,. \end{cases}
\ee
Similarly, we may define an operator
\be \hat P_N :=  \tfrac12\big(\mb{G}_{N+2,i}\hat Z_i+\mb{G}_{N+2,i}'\hat Z_i'+\mb{G}_{N+2,i}''\hat Z_i''\big)
\ee
corresponding to the longitude. Due to the symplectic structure found in \cite{NZ}, we know that we may complete the set $\{\hat X_1,...,\hat X_N,\hat P_N\}$ to a full canonical basis of the algebra $\hat{\mathcal{A}}_\times$. We do this by adding $N-1$ additional operators $\hat P_I$, which are linear combinations of the $\hat Z$'s, such that
\be \qquad  [\hat P_I,\hat X_j] = \delta_{Ij}\hbar\,,\qquad [\hat P_I,\hat P_j]=[\hat X_I,\hat X_j]=0\,,\qquad 1\leq I,j\leq N\,.\ee

The operators $\hat X_I,\hat P_I$ have a simple interpretation in terms of a generalized Neumann-Zagier datum. Namely, if we complete $(\mb A\;\mb B)$ and the rows $C,D$ (of Section \ref{sec.uflat}) to a full symplectic matrix $\left(\begin{smallmatrix} \mb A & \mb B \\ \mb C & \mb D \end{smallmatrix}\right)$, then
\be \label{XPsymp}
 \begin{pmatrix} \hat X \\ \hat P \end{pmatrix} = \begin{pmatrix} \mb A & \mb B \\ \mb C & \mb D \end{pmatrix}\begin{pmatrix} \hat Z \\ \hat Z''\end{pmatrix} - \big(i\pi+\tfrac\hbar2\big)\begin{pmatrix} \eta \\ \eta_P \end{pmatrix}\,.
\ee
Here $\eta$ is precisely the vector of $N$ integers introduced in \eqref{defn}; while $\eta_P = (*,...,*,\eta_\lambda)$, with $\eta_\lambda$ from \eqref{defCD}. The first $N-1$ entries of $\eta_P$ depend on the precise completion of the canonical basis (or the symplectic matrix), and ultimately drop out of the gluing construction.

\subsection{Quantum reduction}

Classically, in order to glue we would want to set the $N-1$ edge constraints $X_I\to 0$, and the meridian $X_N\to 2u$. In Section \ref{sec.sympglue}, these functions were actually used as moment maps to perform a symplectic reduction. Now we should do the same thing quantum mechanically.
In order to reduce the product wavefunction $\CZ_\times(Z_1,...,Z_N)$ of \eqref{prodH} to the final wavefunction $\CZ_M(u)$ of the glued manifold $M$, we must transform the wavefunction to a representation (or ``polarization'') in which the operators $\hat X_I$ act diagonally (by multiplication). In this representation, the wavefunction depends explicitly on the $X_I$. The ``reduction'' then simply requires fixing $X_I\to(0,...,0,2u)$. Schematically,
\be \label{redscheme}
\CZ_\times(Z_1,...,Z_N)\overset{\text{transform}}{\longmapsto} \wt\CZ_\times(X_1,...,X_N) \;\;\overset{\text{fix}}{\mapsto}\;\; \CZ_M(u) = \wt\CZ_\times(0,...,0,2u)\,.
\ee

The transformation from $\CZ_\times$ to $\wt \CZ_\times$ is accomplished --- formally --- with the Weil representation $\mathcal{R}$ of the affine symplectic group \cite{Shale, Weil}. In particular, we need $\mathcal{R}(\alpha)$ for the affine symplectic transformation $\alpha$ in \eqref{XPsymp}. In Section 6 of \cite{D1}, it was discussed in detail how to find $\mathcal{R}(\alpha)$ by factoring the matrix of \eqref{XPsymp} into generators. Then, for example, an ``$S$--type'' element of the symplectic group acts via Fourier transform
\be \label{WeilS}
 \mathcal{R}\left(\left(\begin{smallmatrix} 0 & -I \\ I & 0\end{smallmatrix}\right)\right):\; f(Z)\;\mapsto\;\tilde f(W)=\int\frac{d^N\!Z}{(2\pi i\hbar)^{N/2}}\, e^{\tfrac1\hbar Z\cdot W}f(Z)\,,
\ee
whereas a ``$T$--type'' element acts as multiplication by a quadratic exponential
\be \label{WeilT}
\mathcal{R}\left(\left(\begin{smallmatrix} I & 0 \\ \mb T & I\end{smallmatrix}\right)\right):\; f(Z)\;\mapsto\;\tilde f(W)=e^{\frac1{2\hbar}W^T\mb TW}f(W)\,.
\ee
Affine shifts act either by translation or multiplication by a linear exponential.

In the present case, there is a convenient trick that allows us to find $\mathcal{R}(\alpha)$ without decomposing $\alpha$ into generators. We assume that the block $\mb B$ of the symplectic matrix is nondegenerate, since we know we can always choose a quad type with this property. For the moment, let us also suppose that the affine shifts vanish, $\eta=\eta_P=0$. Then the Weil action is
\be \label{simpleRa}
\mathcal{R}(\alpha)\,:\; \CZ_\times(Z)\;\mapsto\; \wt\CZ_\times(X) = \frac{1}{\sqrt{\det\mb B}}\int\frac{d^N\!Z}{(2\pi i\hbar)^{N/2}}\,e^{\tfrac1{2\hbar}\left( X\cdot \mb D\mb B^{-1}X-2Z\cdot\mb B^{-1}X+Z\cdot\mb B^{-1}\mb AZ\right)}\CZ_\times(Z)\,.
\ee
In particular, it can easily be verified that this correctly intertwines an action of operators $(\hat Z_i=Z_i,\,\hat Z_i''=\hbar\pd_{Z_i})$ on $\CZ_\times(Z)$ with an action of operators $(\hat X_I=X_I,\,\hat P_I=\hbar \pd_{X_I})$ on $\wt\CZ_\times(X)$. For example,
\begin{align*} & \int d^N\!Z\,e^{\tfrac1{2\hbar}\left( X\cdot \mb D\mb B^{-1}X-2Z\cdot\mb B^{-1}X+Z\cdot\mb B^{-1}\mb AZ\right)}\big(\mb A \hat Z+\mb B\hat Z''\big)\CZ_\times(Z) \\
&\qquad = \int d^N\!Z\,e^{\tfrac1{2\hbar}\left( X\cdot \mb D\mb B^{-1}X-2Z\cdot\mb B^{-1}X+Z\cdot\mb B^{-1}\mb AZ\right)}\big(\mb A Z+\hbar\, \mb B\pd_Z\big)\CZ_\times(Z) \\
&\qquad = \int d^N\!Z \Big[\big(\mb A Z-\hbar\, \mb B\pd_Z\big)e^{\tfrac1{2\hbar}\left( X\cdot \mb D\mb B^{-1}X-2Z\cdot\mb B^{-1}X+Z\cdot\mb B^{-1}\mb AZ\right)}\Big]\CZ_\times(Z) \\
&\qquad = \int d^N\!Z\,Xe^{\tfrac1{2\hbar}\left( X\cdot \mb D\mb B^{-1}X-2Z\cdot\mb B^{-1}X+Z\cdot\mb B^{-1}\mb AZ\right)}\CZ_\times(Z) \\
&\qquad = \hat X\,\wt \CZ_\times(X)\,.
\end{align*}

Non-zero affine shifts $\eta$ and $\eta_P$ further modify the result to
\begin{align*} \wt\CZ_\times(X) = &\frac{1}{\sqrt{\det\mb B}}\int\frac{d^N\!Z}{(2\pi i\hbar)^{N/2}}\,
\exp\bigg[-\frac1\hbar X\cdot \Big(i\pi+\tfrac\hbar 2\Big)\eta_P\,+\\
&\qquad \frac1{2\hbar}\bigg( \Big(X+\Big(i\pi+\tfrac\hbar 2\Big)\eta\Big)\cdot \mb D\mb B^{-1}\Big(X+\Big(i\pi+\tfrac\hbar 2\Big)\eta\Big)-2Z\cdot\mb B^{-1}\Big(X+\Big(i\pi+\tfrac\hbar 2\Big)\eta\Big)\\
&\qquad\qquad +Z\cdot\mb B^{-1}\mb AZ\bigg)\bigg]
\CZ_\times(Z)\,,
\end{align*}
and then, after setting $X\to 2\bm u = (0,...,0,2u)$ as in \eqref{redscheme}, we find
\begin{align} \CZ_M(u) = \frac{1}{\sqrt{\det\mb B}}&\int\frac{d^N\!Z}{(2\pi i\hbar)^{N/2}}\,
\exp\bigg[-\frac1\hbar (2\pi i+\hbar)\eta_\lambda u \,+ \notag \\
& \frac1{2\hbar}\bigg( \Big(2\bm{u}+\Big(i\pi+\tfrac\hbar 2\Big)\eta\Big)\cdot \mb D\mb B^{-1}\Big(2\bm{u}+\Big(i\pi+\tfrac\hbar 2\Big)\eta\Big)-2Z\cdot\mb B^{-1}\Big(2\bm{u}+\Big(i\pi+\tfrac\hbar 2\Big)\eta\Big)
\notag \\
&\qquad +Z\cdot\mb B^{-1}\mb AZ\bigg)\bigg]
\prod_{i=1}^N\psi_\hbar(Z_i)\,. \label{ZMu} 
\end{align}
This is the partition function of the one-cusped manifold $M$, modulo a multiplicative ambiguity of the form $\exp\big[\tfrac{\pi^2}6a+\tfrac{i\pi}{4}b+\tfrac{1}{24}c\big]$ for $a,b,c\in \Z$, which we will say more about in Section \ref{app.norm}. By \emph{construction}, this partition function is annihilated by the quantum $\hat A$-polynomial of $M$.

\subsection{Introducing a flattening}

In order to obtain the state integral \eqref{uSIM} appearing in the paper, we can introduce a generalized flattening (as in Section \ref{sec.uflat}) and use it to simplify \eqref{ZMu}. Note that the discrete-faithful state integral \eqref{SIM} follows immediately from \eqref{uSIM} upon setting $\bm u=(0,...,0,u)\to0$.

Suppose, then, that we have integers $(f,f'')$ that satisfy
\be 
\lbl{eq.nuPinteger}
\begin{pmatrix} \mb A & \mb B \\ \mb C & \mb D \end{pmatrix}\begin{pmatrix} f \\ f'' \end{pmatrix} = \begin{pmatrix} \eta \\ \eta_P \end{pmatrix}\,, \ee 
for some $\eta_P$ whose last entry is $\eta_\lambda$. 
We will \emph{assume} that a completed symplectic matrix 
$\left(\begin{smallmatrix} \mb A & \mb B 
\\ \mb C & \mb D \end{smallmatrix}\right)$ can be chosen in $\Sp(2N,\Z)$ 
rather than in $\Sp(2N,\BQ)$. In that case, since $f$ and $f''$ are vectors
with integer entries, it follows that $\eta_P \in \BZ^N$. 
%
%
Then, 
\begin{align*} -\eta_\lambda u +\eta\cdot \mb D\mb B^{-1} \bm u  &=
 -(\mb Cf+\mb D f'')\cdot \bm u+(\mb D^T\mb Af+\mb D^T \mb B f'')\cdot \mb B^{-1}\bm u 
 = f\cdot\mb B^{-1}\bm u\,,
\end{align*}
where we used the symplectic identities $\mb D^T\mb B=\mb B^T\mb D$ and $\mb D^T\mb A=I+\mb B^T\mb C$; and
\begin{align*} \mb \eta \cdot \mb D\mb B^{-1}\eta &= \eta\cdot \mb D(f''+\mb B^{-1}\mb Af) = \eta\cdot (\mb Df''+\mb Cf+\mb B^{-1,T}f) = f\cdot \mb B^{-1}\eta +\nu\cdot \nu_P \\
&= f\cdot \mb B^{-1}\eta\qquad (\text{mod}\;\Z) \end{align*}
in a similar way. These relations allow us to write the state integral 
\eqref{ZMu} as
\be \label{appSIM}
\begin{split}
\hspace{-0in}\CZ_M(u) &= \frac{1}{\sqrt{\det\mb B}} \,\, \cdot \\
 & \int \frac{d^N\!Z}{(2\pi \hbar)^{N/2}}\, e^{\tfrac{1}{\hbar}\Big[2\bm u\cdot \mb D\mb B^{-1}\bm u+ (2\pi i+\hbar)f\cdot \mb B^{-1}\bm u+\tfrac12\big(i\pi+\tfrac\hbar2\big)^2f\cdot\mb B^{-1}\eta-Z\mb B^{-1}\big(2\bm u+\big(i\pi+\tfrac\hbar 2\big)\eta\big)\Big]} 
\prod_{i=1}^N\psi_\hbar(Z_i)\,,
\end{split}
\ee
just as in \eqref{uSIM}. (We drop a factor of $\sqrt i$ from the measure, since it can be absorbed in the overall normalization ambiguity.)

\subsection{Normalization and invariance}
\label{app.norm}

The normalization of Chern-Simons state integrals has always been a subtle issue. For the integral of \cite{D1}, ambiguities in the normalization come from two sources: the projectivity of the Weil representation, and the incomplete invariance of the integral (even formally) under a change of ``quad type'' and a 2--3 move.

Let us consider the Weil representation first. We will assume that all symplectic matrices are in $\Sp(2N,\Z)$, and that all shifts involve integers (like $\eta$ and $\eta_P$) times $i\pi+\tfrac\hbar 2$. This assumption (which, again, is only an observed property) allows us to improve on the estimates of \cite{D1} (\cf\ Eqn.(6.6) there).  The Weil representation becomes a projective unitary representation of $\ISp(2N,\Z)\simeq \Sp(2N,\Z)\ltimes \big[\big(i\pi+\tfrac\hbar2\big)\Z\big]^{2N}$ on $L^2(\BR^N)$, for $\hbar$ pure imaginary. Our Hilbert space $\CH_\Delta^{\otimes 2N}$ is very close to $L^2(\BR^N)$, so we may hope that the Weil representation is also unitary projective there. The most severe projective ambiguity arises from a violation of expected commutation relations between shifts and $T$-type transformations such as \eqref{WeilT}. This leads to projective factors of the form
\be \label{Weilproj} \exp\Big[\frac{1}{2\hbar}\Big(i\pi+\frac\hbar 2\Big)^2a\Big]=\exp\Big[\Big(-\frac{\pi^2}{2\hbar}+\frac{i\pi}{2}+\frac{\hbar^2}{8}\Big)a\Big]\,,\qquad a\in \Z\,.
\ee
With the exception of factors like this, unitarity with respect to the norm 
$$
|\!|f|\!|^2=\int \frac{d^N\!Z}{(\pm2\pi i\hbar)^{N/2}}|f(Z)|^2
$$ 
may be used to normalize Weil transformations. For example, the factor $\big[(2\pi i\hbar)^N\det \mb B\big]^{-1/2}$ in \eqref{simpleRa} follows easily from formal manipulations on the integral transformation to demonstrate unitarity.

The lack of complete invariance under a change of quad type (cyclic permutation invariance) and a 2--3 move can also ruin the normalization of the state integral. The change of quad type was analyzed, formally, in Section 6.2.1 of \cite{D1}. A cyclic permutation of a tetrahedron is accomplished by an affine version of the element $ST\in \Sp(2N;\Z)$, under the Weil representation. The single-tetrahedron wavefunction transforms 
as
\be 
\psi_\hbar(Z) \;\mapsto\; \int \frac{dZ}{\sqrt{2\pi i\hbar}}\,e^{\tfrac1{2\hbar}\big(Z^2+2ZZ'-(2\pi i+\hbar)Z\big)}\psi_\hbar(Z) \;=\; e^{\tfrac{\pi^2}{6\hbar}\pm\tfrac{i\pi}{4}-\tfrac\hbar{24}}\psi_\hbar(Z')\,.
\ee
The last equality follows from the Fourier transform of the quantum dilogarithm \cite{FKVolkov,PT}. This shows that the tetrahedron wavefunction is \emph{invariant} under permutations, up to a factor
\be \label{quadproj}
\exp\Big[\Big(\frac{\pi^2}{6\hbar}\pm\frac{i\pi}{4}-\frac\hbar{24}\Big)a\Big]\,,\qquad a\in \Z\,.\ee

The analysis of the 2--3 move is slightly more involved. It was done in terms of operator algebra in \cite{D1}, and then explained in terms of wavefunctions in Section 6.2 of \cite{DGG1}. The main idea is that a 2--3 move can be done locally during the gluing procedure, by performing a formal, ``local'' transformation on the state integral. The crucial property involved is the Ramanujan-like identity for the quantum dilogarithm \cite{FKVolkov, PT}, which expresses three quantum dilogarithms as an integral of two; for example,
\begin{align} \notag \psi_\hbar(W_1')\psi_\hbar(W_2')\psi_\hbar(W_3')\big|_{W_1'+W_2'+W_3'=2\pi i+\hbar} \\ &\hspace{-1.5in} \;\sim\; \int \frac{dZ}{\sqrt{2\pi i\hbar}}\, e^{\tfrac{1}{2\hbar}\big(Z^2 +2 W_2'Z-(2\pi i+\hbar)(W_1'+W_2'+Z)\big)} \psi_\hbar(-Z)\psi_\hbar(Z-W_1')
\end{align}
which holds up to a factor that is again of the type \eqref{quadproj}. 

Putting together all three effects, we find that we might be able to control the overall normalization of the state integral up to a factor of the form
\be \label{SIMfactor}
\exp\Big[\frac{\pi^2}{6\hbar}a+\frac{i\pi}{4}b+\frac{\hbar}{24}c\Big]\,,\qquad a,b,c\in \Z\,.
\ee


\section{Computer implementation and computations}
\lbl{app.compute}

An enhanced Neumann-Zagier datum is a tuple $(z,\mb A,\mb B,f)$ attached
to a regular ideal triangulation of a cusped hyperbolic manifold $M$.
The program \texttt{SnapPy} \cite{SnapPy} in its {\tt python} and 
{\tt sage} implementation computes the gluing matrices 
$\mb G,\mb G',\mb G''$ of Sections \ref{sub.flat} and \ref{sec.uflat}; 
and therefore it can easily compute an enhanced Neumann-Zagier datum 
$\widehat\beta_\CT=(z,\mb A,\mb B,f)$. The shape parameters $z$
are algebraic numbers that are computed numerically to arbitrary precision
(eg, 10000 digits) or exactly as algebraic numbers.

A {\tt Mathematica} module of the authors computes (numerically or exactly)
the $n$-loop invariants $S_{\CT,n}$ for $n=0,2,3$ as well as our torsion
$\tau_{\CT}$ given as input the Neumann-Zagier datum.
As an example, consider the hyperbolic knot $\mb{9_{12}}$ with volume 
$8.836642343\dots$ and the {\tt SnapPy} ideal triangulation with 
$10$ tetrahedra. Its invariant trace field $E_{\mb{9_{12}}}$
is $\BQ(x)$ where 
$x=-0.06265158\dots + i \,1.24990458\dots$ is a root of 
\begin{align*}
&x^{17} - 8 x^{16} + 32 x^{15} - 89 x^{14} + 195 x^{13} - 353 x^{12} + 542 x^{11} - 
719 x^{10} + 834 x^9  \\
& \qquad - 851 x^8 + 764 x^7 - 605 x^6+ 421 x^5 - 
253 x^4 + 130 x^3 - 55 x^2 + 18 x - 3=0
\end{align*}
$E_{\mb{9_{12}}}$ is of type $[1,8]$ with discriminant $3 \cdot 298171 \cdot 5210119
\cdot 156953399$.
Our torsion is given by
\begin{eqnarray*}
\tau_{\mb{9_{12}}}&=& \frac{1}{2}
\left(15 -7 x -15 x^2 + 55 x^3 -67 x^4 + 81 x^5 -43 x^6 -112 x^7 + 303 x^8 -488
x^9 \right. \\ 
& & \left.
+ 606 x^{10} -595 x^{11} + 464 x^{12} -289 x^{13} + 143 x^{14} -49 x^{15} + 8 x^{16}
\right)
\\
&=& -3.133657804174628986\dots + 14.061239582208047255\dots \, i 
\end{eqnarray*}
The two and three-loop invariants simplify considerably when multiplied
by $\tau_{\mb{9_{12}}}^3$ and $\tau_{\mb{9_{12}}}^6$ respectively and are given by
\begin{eqnarray*}
S_{\mb{9_{12}},2}\,\tau_{\mb{9_{12}}}^3&=& 
\frac{1}{2^6 \cdot 3}
\left(36263 -194718 x + 503316 x^2 -971739 x^3 + 1582041 x^{4} -2152164 x^{5} 
\right. \\ 
& & \left.
+ 2372779 x^{6} -2109742 x^{7} + 1426659 x^{8} -484152 x^{9} -374803 x^{10} 
+ 836963 x^{11} 
\right. \\ 
& & \left.
-859483 x^{12} + 621288 x^{13} -326550 x^{14} + 109607 x^{15} 
-16840 x^{16} \right) \\
&=& 398.62270435384630954\dots + 948.91209325049603870\dots i
\end{eqnarray*}

\begin{eqnarray*}
S_{\mb{9_{12}},3}\tau_{\mb{9_{12}}}^6&=& 
\frac{1}{2^7}
\left(2320213 -19092785 x^{1} + 72589953 x^{2} -186402605 x^{3} 
+ 382362100 x^{4} 
\right. \\ 
& & \left.
-661985976 x^{5} + 982969902 x^{6} -1258919324 x^{7} 
+ 1402544816 x^{8} 
\right. \\ 
& & \left.
-1359436057 x^{9} + 1134208276 x^{10} -803313515 x^{11} 
+ 473961630 x^{12} 
\right. \\ 
& & \left.
-225394732 x^{13} + 80872920 x^{14} -19104127 x^{15} 
+ 2161102 x^{16}
\right)
\\
&=& 71793.64335382669630\dots + 204530.00105728258992\dots i
\end{eqnarray*}
The norm $(N_1,N_2,N_3)=\big(N(\tau_{\mb{9_{12}}}), N(S_{\mb{9_{12}},2}\tau_{\mb{9_{12}}}^3),
N(S_{\mb{9_{12}},3}\tau_{\mb{9_{12}}}^6)\big)$ 
of the above algebraic numbers is given by
{\tiny
\begin{eqnarray*}
N_1&=& \frac{3 \cdot 298171 \cdot 5210119 \cdot 156953399}{2^{17}}
\\
N_2&=& 
\frac{ 173137 \cdot 
2497646101253660962719786587619396709848261029974343408485645575954409}{
2^{102} \cdot 3^{17}}
\\
N_3&=&
\frac{1601979456387778103376978278735985249091224621424605434099548771751970874984335599199394506045984185143}{2^{119}}
\end{eqnarray*}
}
\noindent 
Recall that although $S_{2,\mb{9_{12}}}$ is defined modulo an integer multiple of 
$1/24$, $S_{3,\mb{9_{12}}}$ is defined without ambiguity 
and the numerator $N_3$ is a prime number of $103$ digits.

For a computation of the Reidemeister torsion $\tau^{\Reid}_M$ of the discrete 
faithful representation of a cusped hyperbolic manifold $M$,  
we use a theorem of Yamaguchi \cite{Ya} to identify it with
$$
\tau^\Reid_M=\frac{1}{c_M}\frac{d\tau^\Reid_M(t)}{dt}\left|_{t=1} \right.
$$ 
where $c_M$ is the cusp shape of $M$ and 
$\tau^\Reid_M(t) \in E_M[t^{\pm 1}]$ is the {\em torsion polynomial} of $M$ using the
adjoint representation of $\SL(2,\BC)$. 
Using the {\tt hypertorsion} package of N. Dunfield (see \cite{DFJ}), 
we can compute $\tau^\Reid_M$ as follows:

\begin{verbatim}
cd Genus-Comp
sage:import snappy, hypertorsion

def torsion(manifold, precision=100):
    M = snappy.Manifold(manifold)
    p = hypertorsion.hyperbolic_adjoint_torsion(M, precision)
    q = p.derivative()
    rho = hypertorsion.polished_holonomy(M, precision)
    z = rho.cusp_shape()
    torsion = q(1)/z.conjugate()
    return [M.name(), torsion]
\end{verbatim}
For the above example, we have:

\begin{verbatim}
sage: torsion("9_12",500)
['L105002',  -3.133657804174628986... + 14.061239582208047255...*I]
\end{verbatim}
numerically confirming Conjecture \ref{conj.1loop}.
Further computations gives a numerical confirmation
of Conjecture \ref{conj.1loop} to 1000 digits for all $59924$ hyperbolic 
knots with at most $14$ crossings.


\bibliographystyle{hamsalpha}
\bibliography{biblio}
\end{document}